\documentclass{amsart}
\usepackage{latexsym,amsxtra,amscd,ifthen}
\usepackage{amsfonts}
\usepackage{verbatim}
\usepackage{amsmath}
\usepackage{amsthm}
\usepackage{amssymb}

\numberwithin{equation}{section}

\theoremstyle{plain}
\newtheorem{theorem}{Theorem}[section]
\newtheorem{lemma}[theorem]{Lemma}
\newtheorem{proposition}[theorem]{Proposition}
\newtheorem{corollary}[theorem]{Corollary}

\theoremstyle{definition}
\newtheorem{definition}[theorem]{Definition}
\newtheorem{example}[theorem]{Example}
\newtheorem{convention}[theorem]{Convention}
\newtheorem{remark}[theorem]{Remark}
\newtheorem{question}[theorem]{Question}

\makeatletter              
\let\c@equation\c@theorem 

\makeatother

\DeclareMathOperator{\gldim}{gldim}

\DeclareMathOperator{\Ext}{Ext}
\DeclareMathOperator{\Tor}{Tor}
\DeclareMathOperator{\pdim}{pdim}

\DeclareMathOperator{\gr}{gr}

\DeclareMathOperator{\grmod}{grmod}

\DeclareMathOperator{\GKdim}{GKdim}
\DeclareMathOperator{\End}{End}

\DeclareMathOperator{\Hom}{Hom}
\DeclareMathOperator{\RHom}{RHom}

\newcommand{\newdeg}{\deg'}

\begin{document}

\title{Double Ore Extensions}

\author{James J. Zhang and Jun Zhang}

\address{Department of Mathematics, Box 354350,
University of Washington, Seattle, WA 98195, USA}

\email{(James J. Zhang) zhang@math.washington.edu}
\email{(Jun Zhang) junz@math.washington.edu}

\begin{abstract}
A double Ore extension is a natural generalization of the 
Ore extension. We prove that a connected graded double 
Ore extension of an Artin-Schelter regular algebra is 
Artin-Schelter regular. Some other basic properties such 
as the determinant of the DE-data are studied. Using the 
double Ore extension, we construct 26 families of 
Artin-Schelter regular algebras of global dimension four 
in a sequel paper.
\end{abstract}




\maketitle


\setcounter{section}{-1}
\section{Introduction}
\label{xxsec0}

In 1933 Ore introduced and studied a version of noncommutative 
polynomial ring \cite{Or} which has become one of most basic 
and useful constructions in ring theory. Such a polynomial ring 
is now called {\it Ore extension} \cite[p.37]{MR}. Given an 
algebra $A$, an algebra automorphism $\sigma$ of $A$ and a
$\sigma$-derivation $\delta$ of $A$, the Ore extension of $A$ 
associated to $(\sigma,\delta)$ is obtained by adding a single 
generator $y$ to $A$ subject to the relation 
$$y r=\sigma(r) y+\delta(r)$$
for all $r\in A$. This Ore extension of $A$ is denoted by 
$A[y;\sigma,\delta]$. In this paper we introduce a new 
construction that is obtained by adding two generators $y_1$ and 
$y_2$ simultaneously to $A$. This construction is a natural 
generalization of the Ore extension; and for this reason, it is 
called a {\it double Ore extension} or just {\it double extension}
for short. Note that a double extension is usually different from
an iterated Ore extension that is obtained by forming the Ore 
extension twice (see the algebra in Example \ref{xxex4.2} and 
Proposition \ref{xxprop0.5}(c)). A double extension of $A$ is denoted 
by $A_P[y_1,y_2;\sigma,\delta,\tau]$ and the meanings of 
DE-data $\{P,\sigma,\delta,\tau\}$ will be explained 
soon. Some ideas related to double extensions were used by Patrick 
\cite{Pa} and Nyman \cite{Ny}, however, we have not found 
any literature devoted to a systematic study of this construction.
Our main interest on double extension is to construct more 
Artin-Schelter regular algebras of global dimension four \cite{ZZ}.
The double extension seems quite similar to the Ore extension. 
So we wish to extend some basic properties of 
Ore extensions to double extensions.. Surprisingly, many 
techniques used for Ore extensions are invalid for double extensions. 
For example, we have not been able to prove that a double extension 
of a noetherian ring is noetherian. Many new and complicated 
constraints have to be posted in constructing a double extension. 
General ring-theoretic properties of double extensions are not 
known (see questions in Section 4).

The motivation for this paper is to construct new Artin-Schelter 
regular algebras. It is very important in noncommutative algebraic 
geometry to classify all Artin-Schelter regular algebras of global 
dimension four that are the homogeneous coordinate rings of noncommutative
projective 3-spaces. There has been extensive research in this topic; 
and many non-isomorphic families of regular algebras have been 
discovered in recent years 
\cite{LPWZ, Sk1, Sk2, SS, VV1, VV2, VVW, Va1, Va2}. 
In a sequel paper \cite{ZZ} we use the double extension to 
construct 26 families of Artin-Schelter regular algebras of 
dimension four. One of the main results in \cite{ZZ} is

\begin{theorem}
\cite[Theorem 0.1(a)]{ZZ}
\label{xxthm0.1}
Let $B$ be an Artin-Schelter regular domain of global dimension four
that is generated by four elements of degree 1. If $B$ is a double
extension, then it is strongly noetherian, Auslander regular and 
Cohen-Macaulay.
\end{theorem}

In this paper we study some basic ring-theoretic and homological 
properties of double extensions so we can focus on the other issues 
such as classification in the sequel \cite{ZZ}. As mentioned earlier
double extensions seem much more difficult to study than Ore 
extensions. One result we can prove is the following.

\begin{theorem}
\label{xxthm0.2} 
Let $A$ be an Artin-Schelter regular algebra. If $B$ is a connected 
graded double extension of $A$, then $B$ is Artin-Schelter regular
and $\gldim B=\gldim A+2$. 
\end{theorem}

To get a further idea about a double extension, let us explain 
the associated DE-data $\{P,\sigma,\delta,\tau\}$. Let $\sigma$ 
be an algebra homomorphism $A\to M_2(A)$ where $M_2(A)$ is the 
$2\times 2$-matrix algebra over $A$ and let $\delta$ be a 
$\sigma$-derivation from $A\to A^{\oplus 2}$, where
$A^{\oplus 2}$ is the $2\times 1$-matrix over $A$, satisfying 
$$\delta(rs)=\sigma(r)\delta(s)+\delta(r)s$$
for all $r,s\in A$. The multiplication $\sigma(r)\delta(s)$
makes sense because $A^{\oplus 2}$ is a left $M_2(A)$-module.
Two extra generators $y_1$ and $y_2$ satisfies a ``quadratic'' 
relation
$$y_2y_1=p_{12}y_1y_2+p_{11}y_1^2+\tau_{1}y_1+\tau_{2}y_2+\tau_{0}$$
where $p_{12},p_{11}\in k$ and $\tau_{1},\tau_{2},\tau_{0}\in A$. We 
call $P:=\{p_{12},p_{11}\}$ the parameter and $\tau:=\{\tau_{1},\tau_{2},
\tau_{0}\}$ the tail. There are other conditions relating 
$y_1,y_2$ with $\{P,\sigma,\delta,\tau\}$. For example, we need
$$
\begin{pmatrix} y_1 r\\y_2 r\end{pmatrix} =
\sigma(r)
\begin{pmatrix} y_1\\y_2\end{pmatrix}
+\delta(r)
$$
for all $r\in A$. A complete description of a double extension 
is given from Definition \ref{xxdefn1.3} to Convention 
\ref{xxcon1.6}. A shorter and more abstract definition of a 
connected graded double extension is given in \cite[Definition 1.1
and Lemma 1.4]{ZZ}.

As a consequence of Definition \ref{xxdefn1.3},  we have that 
$p_{12}\neq 0$ and that $\sigma$ is invertible in the sense of 
Definition \ref{xxdefn1.8} (note that the map $\sigma$ can not 
be invertible in the usual sense because it is not surjective). 
Without these extra conditions, $A_{P}[y_1,y_2;\sigma,\delta,\tau]$ 
is called a right double extension. When $\tau$ and $\delta$ are zero, 
the (right) double extension is denoted by $A_{P}[y_1,y_2;\sigma]$. 

The DE-data are useful; for example, we can define 
the determinant of $\sigma$ which plays an essential
role in proving Theorem \ref{xxthm0.2}.
Let $B=A_{P}[y_1,y_2;\sigma,\delta,\tau]$ be a (right) double 
extension. The {\it determinant} of $\sigma$ is defined to be
$$\det \sigma: r\mapsto -p_{11}\sigma_{12}(\sigma_{11}(r))
+\sigma_{22}(\sigma_{11}(r))-p_{12}\sigma_{12}(\sigma_{21}(r))$$
for all $r\in A$, where $\sigma(r)=
\begin{pmatrix} \sigma_{11}(r)&\sigma_{12}(r)\\
                 \sigma_{21}(r)&\sigma_{22}(r)
\end{pmatrix}$. This formula reminds us the quantum determinant
of a quantum $2\times 2$-matrix; and when $P=(1,0)$ it becomes
$$\det \sigma=\sigma_{22}\circ \sigma_{11}-\sigma_{12}\circ
\sigma_{21}$$
which agrees with the usual determinant in linear
algebra (but $\det \sigma\neq \sigma_{11}\circ \sigma_{22}-
\sigma_{12}\circ\sigma_{21}$ in general). The following result 
characterizes the invertibility of $\sigma$.

\begin{proposition}[Proposition \ref{xxprop2.1}]
\label{xxprop0.3} Let $B=A_{P}[y_1,y_2;\sigma,\delta,\tau]$ be
a right double extension.
\begin{enumerate}
\item
$\det \sigma$ is an algebra endomorphism of $A$.
\item
Suppose that $p_{12}\neq 0$ and that $B$ is a connected graded 
right double extension. Then $\det \sigma$ is invertible if and 
only if $\sigma$ is invertible in the sense of Definition 
\ref{xxdefn1.8}.
\end{enumerate}
\end{proposition}

The determinant $\det \sigma$ can also be recovered by the
cohomology.

\begin{theorem}[Theorem \ref{xxthm2.2}]
\label{xxthm0.4} Let $B=A_{P}[y_1,y_2;\sigma]$ be a double
extension. Identifying $A$ with the factor ring $B/(y_1,y_2)$. 
Then 
$$\Ext^i_B(A,B)=\begin{cases} {^{\det \sigma}A} &i=2\\
0&i\neq 2\end{cases}$$
where the left and right $A$-actions on the $A$-bimodule 
$^{\det\sigma}A$ is defined by
$$l*a*r=(\det\sigma (l))ar$$
for all $l,a,r\in A$.
\end{theorem}

Since the double extension is a new construction we feel that
it is worth to study some simple examples in some details.
Let $h$ be a nonzero scalar in $k$ and let $B(h)$ be the algebra
generated by four degree 1 elements $x_1,x_2,y_1,y_2$ and 
subject to the six quadratic relations
$$\begin{aligned}
x_2x_1&=-x_1x_2,\\
y_2y_1&=-y_1y_2,\\
y_1x_1&=hx_1y_1+hx_2y_1+hx_1y_2,\\
y_1x_2&=hx_1y_2 ,\\
y_2x_1&=hx_2y_1,\\
y_2x_2&=-hx_2y_1-hx_1y_2+hx_2y_2.
\end{aligned}
$$

\begin{proposition}
\label{xxprop0.5} Let $B(h)$ be the algebra defined as above.
\begin{enumerate}
\item
$B(h)$ is a double extension $A_{P}[y_1,y_2;\sigma]$ where 
$A=k\langle x_1,x_2\rangle/(x_1x_2+x_2x_1)$ and $P=(-1,0)$.
\item
It is an Artin-Schelter regular, Auslander regular, 
Cohen-Macaulay, Koszul and strongly noetherian domain.
\item
It is neither an Ore extension nor a normal extension
of a 3-dimensional Artin-Schelter regular algebra.
\item
$B(h)$ is PI if and only if $h$ is a root of unity if and 
only if the automorphism $\det\sigma$ has finite order.
\item
The quotient division ring of $B(h)$ is generated by 
two elements.
\end{enumerate}
\end{proposition}

Other properties of $B(h)$ can be found in Section 4. 

Here is an outline of the paper. We give the definition and
some remarks in Section 1. The proofs of Proposition 
\ref{xxprop0.3} and Theorem \ref{xxthm0.4} (depending on 
a lot of computation) are given in Section 2. 
The proof of Theorem \ref{xxthm0.2} is given in Section 3. 
We like to remark that Theorem \ref{xxthm0.4} is a key step 
for the proof of Theorem \ref{xxthm0.2}. Section 4 contains 
some examples and some basic questions.

A multi-variable version of the double extension can be defined 
similar to Definition \ref{xxdefn1.3}. We expect that multi-variable 
extensions are useful for constructing higher dimensional 
Artin-Schelter regular algebras. 

\section{Definitions} 
\label{xxsec1}

Throughout $k$ is a commutative base field. For convenience we
also assume that $k$ is algebraically closed. Everything is 
over $k$; in particular, an algebra or a ring is a $k$-algebra. 
Here is the definition of the original Ore extension.

\begin{definition}
\label{xxdefn1.1}
Let $A$ be an algebra with automorphism $\sigma$.
Let $\delta$ be a $\sigma$-derivation that satisfies 
$$\delta(ab)=\sigma(a)\delta(b)+\delta(a)b$$ 
for all $a,b\in A$. The {\it Ore extension} of 
$A$ associated with data $\{\sigma,\delta\}$ is a ring $B$, 
containing $A$ as a subring, generated by elements 
in $A$ and a new variable $y$ and subject to the 
relation
\begin{equation}
\label{E1.1.1}
ya=\sigma(a)y+\delta(a)
\tag{E1.1.1}
\end{equation}
for all $a\in A$. The Ore extension ring $B$ is
denoted by 
$A[y;\sigma,\delta]$. 
\end{definition}

\begin{remark}
\label{xxrem1.2}
The relation \eqref{E1.1.1} is different from
the one in the definition of the Ore extension given 
in \cite[p.15]{MR}. In fact the Ore extension 
$A[y;\sigma,\delta]$ in Definition \ref{xxdefn1.1} 
is isomorphic to the Ore extension $A[y;\sigma',\delta']$
of \cite[p.15]{MR} with  
$(\sigma',\delta')=(\sigma^{-1},-\delta\sigma^{-1})$. 
Therefore two definitions are equivalent. The reason we 
choose \eqref{E1.1.1} instead of the one in 
\cite[Chapter 1]{MR} is that a formula
similar to \eqref{E1.1.1}, namely \eqref{R2}, is 
more natural when we apply Bergman's diamond 
lemma in the double extension case. 
\end{remark}

We refer to \cite[Chapter 1]{MR} for some basic
properties of Ore extensions. One way to characterize
Ore extensions is that $A[y;\sigma,\delta]$ is
isomorphic to free modules $\bigoplus_{n\geq 0} Ay^n$ 
and $\bigoplus_{n\geq 0} y^n A$ as left and right 
$A$-modules respectively. We use this 
characterization to give an ``abstract''  definition 
of double extension without using $\sigma$ and 
$\delta$. Since our main interest is about connected 
graded rings, we will restrict our attention to the 
graded case in later sections. 

\begin{definition}
\label{xxdefn1.3} 
Let $A$ be an algebra and $B$ be another algebra 
containing $A$ as a subring. 
\begin{enumerate}
\item
We say $B$ is a {\it right double extension} of $A$ 
if the following conditions hold:
\begin{enumerate}
\item[(ai)]
$B$ is generated by $A$ and two new variables 
$y_1$ and $y_2$.
\item[(aii)]
$\{y_1,y_2\}$ satisfies a relation
\begin{equation}
\label{R1}
y_2y_1=p_{12}y_1y_2+p_{11}y_1^2+\tau_{1}y_1+
\tau_{2}y_2+\tau_{0}
\tag{R1}
\end{equation}
where $p_{12},p_{11}\in k$ and $\tau_{1}, \tau_{2}, 
\tau_{0} \in A$. 
\item[(aiii)]
As a left $A$-module, $B=\sum_{n_1,n_2\geq 0} 
Ay_1^{n_1}y_2^{n_2}$ and it is a left free 
$A$-module with a basis $\{y_1^{n_1}y_2^{n_2}
\;|\; n_1\geq 0,n_2\geq 0\}$.
\item[(aiv)]
$y_1 A+y_2 A\subseteq A y_1 + A y_2 + A$
\end{enumerate}
In the graded case it is required that all relations
of $B$ are homogeneous with assignment $\deg y_1>0$
and $\deg y_2>0$. Let $P$ denote the set of scalar 
parameters $\{p_{12},p_{11}\}$ and  let $\tau$ denote 
the set $\{\tau_{1},\tau_{2},\tau_{0}\}$. We 
call $P$ the {\it parameter} and $\tau$ the 
{\it tail}. 
\item
We say $B$ is a {\it left double extension} of $A$ 
if the following conditions hold:
\begin{enumerate}
\item[(bi)]
$B$ is generated by $A$ and two variables 
$y_1$ and $y_2$.
\item[(bii)]
$\{y_1,y_2\}$ satisfies 
\begin{equation}
\label{L1}
y_1y_2=p_{12}'y_2y_1+p_{11}'y_1^2+y_1\tau_{1}'+
y_2\tau_{2}'+\tau_{0}'
\tag{L1}
\end{equation}
where $p_{12}',p_{11}'\in k$ and $\tau_{1}',\tau_{2}', 
\tau_{0}' \in A$. 
\item[(biii)]
As a right $A$-module, $B=\sum_{n_1,n_2\geq 0} 
y_2^{n_1}y_1^{n_2}A$ and it is a right free 
$A$-module with a basis $\{y_2^{n_1}y_1^{n_2}
\;|\; n_1\geq 0,n_2\geq 0\}$.
\item[(biv)]
$Ay_1 +Ay_2 A\subseteq y_1 A +  y_2 A + A$.
\end{enumerate}
\item
We say $B$ is a {\it double extension} if it is a 
left and a right double extension of $A$ with the same 
generating set $\{y_1,y_2\}$.
\end{enumerate}
\end{definition}

We admit that our definition of a double extension is neither 
most general nor ideal, but it works very well in \cite{ZZ} 
where we only consider connected graded regular algebras 
generated in degree 1. Hopefully an improvement of the definition 
will be found when we study other classes of noncommutative rings. 

\begin{remark}
\label{xxrem1.4} 
\begin{enumerate}
\item
In many examples presented in later sections, both $A$ 
and $B$ are connected graded and $\deg y_1=\deg y_2=1$. 
\item
Up to a linear transformation of the vector space
$ky_1+ky_2$, the parameter set $\{p_{12},p_{11}\}$ can 
be assumed to be one of the following: 
$\{0,0\},\{1,1\}$ and $\{p,0\}$ for some $p\neq 0$. Such 
a linear transformation does not affect the definition of 
a right double extension.
\item
The relation (R1) could be made as general  as
\begin{equation}
\label{R1'}
p_{22}y^2+p_{21}y_2y_1+p_{12}y_1y_2+p_{11}y_1^2=
\tau_{1}y_1+\tau_{2}y_2+\tau_{0}.
\tag{R1'}
\end{equation}
Assume $p_{22}p_{11}-p_{12}p_{21}\neq 0$. After a
linear transformation, we may assume that
$\{p_{22},p_{21},p_{12},p_{11}\}$ is either 
$\{0,1,p_{12},0\}$ or $\{0,1,-1,-1\}$. Therefore
the equation (R1') under the condition 
$p_{22}p_{11}-p_{12}p_{21}\neq 0$ is equivalent to 
the equation (R1) under the condition $p_{12}\neq 0$.
\item
If $B$ is a right double extension as in Definition 
\ref{xxdefn1.3}(a), we do not require that $p_{12}\neq 0$. 
If $B$ is a double extension of $A$, then, by comparing
(R1) with (L1), $p_{12}p_{12}'=1$ and hence $p_{12}\neq 0$.
\item
If $B$ is a double extension, then $Ay_1+Ay_2+A=y_1A+y_2A+A$ 
and it is a free $A$-module of rank 3 on both sides.
\end{enumerate}
\end{remark}

We recall a basic property of an Ore 
extension which indicates that the condition in
Definition \ref{xxdefn1.3}(aiii) is reasonable.

\begin{lemma}
\label{xxlem1.5}
Let $B=A[y_1;\sigma_1,\delta_1]$ be an Ore 
extension of $A$ and $C=B[y_2;\sigma_2,\delta_2]$ 
be an Ore extension of $B$.
\begin{enumerate}
\item
$B$ is a free left $A$-module with a basis
$\{y_1^n\}_{n\geq 0}$.
\item 
$C$ is a free left $A$-module with a basis 
$\{y_1^{n_1}y_2^{n_2}\}_{n_1,n_2\geq 0}$.
\end{enumerate}
\end{lemma}

The ring $C=(A[y_1;\sigma_1,\delta_1])[y_2;\sigma_2,\delta_2]$
in the above lemma is called an {\it iterated Ore extension}
of $A$. In general $C$ is not a (right) double extension 
in the sense of Definition \ref{xxdefn1.3}(a) because 
$C$ might not have a relation of the form \eqref{R1}.
If $\sigma_2$ is chosen properly so that (R1) holds 
for $C$, then $C$ becomes a (right) double extension. So many 
double extensions are iterated Ore extensions and vice versa. 

For our computation it is more useful to have an explicit 
description than an abstract definition. We rewrite the 
condition in Definition \ref{xxdefn1.3}(aiv) as follows:
\begin{equation}
\label{R2}
\begin{pmatrix} y_1\\y_2\end{pmatrix} r
:= \begin{pmatrix} y_1 r\\y_2 r\end{pmatrix} =
\begin{pmatrix} \sigma_{11}(r)& \sigma_{12}(r)\\
\sigma_{21}(r)&\sigma_{22}(r)\end{pmatrix} 
\begin{pmatrix} y_1\\y_2\end{pmatrix}
+\begin{pmatrix} \delta_1(r)\\ \delta_2(r)
\end{pmatrix}
\tag{R2}
\end{equation}
for all $r\in A$. 

Write $\sigma(r)=\begin{pmatrix} \sigma_{11}(r)&
\sigma_{12}(r)\\
\sigma_{21}(r)&\sigma_{22}(r)\end{pmatrix} $ and
$\delta(r)=
\begin{pmatrix} \delta_1(r)\\ \delta_2(r)
\end{pmatrix}$. Then $\sigma$ is a $k$-linear map 
from $A$ to $M_2(A)$ where $M_2(A)$ is the 
$2\times 2$-matrix algebra over $A$; and 
$\delta$ is a $k$-linear map from $A$ to the 
column $A$-module 
$A^{\oplus 2}:=\begin{pmatrix} A\\A\end{pmatrix}$. 
The equation (R2) can be written as
$$\begin{pmatrix} y_1\\y_2\end{pmatrix} r
=\sigma(r)
\begin{pmatrix} y_1\\y_2\end{pmatrix}
+\delta(r).$$
This equation is a generalization of \eqref{E1.1.1}.
Suppose $\sigma: A\to M_2(A)$ is an algebra homomorphism, 
namely, $\sigma(rs)=\sigma(r)\sigma(s)$ for all 
$r,s\in A$. A $k$-linear
map $\delta: A\to A^{\oplus 2}$ is called
a {\it $\sigma$-derivation} if 
$$\delta(rs)=\sigma(r)\delta(s)+\delta(r) s$$
for all $r,s\in A$.

\begin{convention}
\label{xxcon1.6}
\begin{enumerate}
\item
By (R2), $\sigma$ and $\delta$ are uniquely determined and 
$\sigma$ is a $k$-linear map from $A$ to $M_2(A)$ and $\delta$ 
is a $k$-linear map from $A$ to $A^{\oplus 2}$. Together with 
Definition \ref{xxdefn1.3}(a), all symbols of 
$\{P,\sigma,\delta,\tau\}$ are defined now. When everything is 
understood, a right double extension or a double extension $B$ 
is also denoted by $A_P[y_1,y_2;\sigma,\delta,\tau]$. By this
notation, we are working with a right double extension though
$B$ can be also a (left) double extension. 
\item
By the next lemma, when $B=A_P[y_1,y_2;\sigma,\delta,\tau]$ 
is a right double extension, then $\sigma$ is an algebra 
homomorphism and $\delta$ a $\sigma$-derivation. Recall from 
Definition \ref{xxdefn1.3}(a) $P$ is called a {\it parameter} 
and $\tau$ is called a {\it tail}. We now call $\sigma$ a 
{\it homomorphism}, $\delta$ a {\it derivation} and the 
collection $\{P,\sigma,\delta,\tau\}$ {\it DE-data} of $A$.
\item
If $\delta$ is a zero map and $\tau$ consists of
zero elements, then the right double extension is denoted 
by $A_{P}[y_1,y_2;\sigma]$. We call $A_{P}[y_1,y_2;\sigma]$ 
a {\it trimmed (right) double extension}.
\end{enumerate}
\end{convention}

\begin{lemma}
\label{xxlem1.7}
Let $B=A_P[y_1,y_2;\sigma,\delta,\tau]$ be a right 
double extension of $A$. Suppose that $\{\sigma,\delta\}$
are defined by \textup{(R2)}. Then the following hold.
\begin{enumerate}
\item
$\sigma$ is an algebra homomorphism $A\to M_2(A)$.
\item
$\delta: A\to A^{\oplus 2}$ is a 
$\sigma$-derivation.
\item
If $\sigma: A\to M_2(A)$ is an algebra homomorphism
and $\delta:
A\to A^{\oplus 2}$ is a $\sigma$-derivation, then
\textup{(R2)}
holds for all elements $r\in A$ if and only if it
holds
for a set of generators.
\end{enumerate}
\end{lemma}

\begin{proof} (a,b) By (R2) we have 
$$\begin{pmatrix} y_1\\y_2\end{pmatrix} (rs)
=\sigma(rs)
\begin{pmatrix} y_1\\y_2\end{pmatrix}
+\delta(rs).$$ 
Using associativity we can also express 
$\begin{pmatrix} y_1\\y_2\end{pmatrix} (rs)$
as
$$(\begin{pmatrix} y_1\\y_2\end{pmatrix}r)s
=(\sigma(r)
\begin{pmatrix} y_1\\y_2\end{pmatrix}
+\delta(r))s
=\sigma(r)\sigma(s)
\begin{pmatrix} y_1\\y_2\end{pmatrix}
+(\sigma(r)\delta(s)+\delta(r)s).$$
By Definition \ref{xxdefn1.3}(aiii) $Ay_1+Ay_2+A$
is a free module $Ay_1\oplus Ay_2\oplus A$
with a basis $\{y_1,y_2,1\}$. 
The assertions follow by comparing the 
coefficients in two different expressions of
$\begin{pmatrix} y_1\\y_2\end{pmatrix} (rs)$.

(c) Suppose (R2) holds for a set of generators 
of $A$. Let $S$ be the set of elements in $A$ 
such that (R2) holds. Then $S$ is a $k$-linear 
subspace of $A$ closed under multiplication 
by the facts that $\sigma$ is an algebra 
homomorphism and that $\delta$ is a 
$\sigma$-derivation. So $S=A$ since $S$
contains a set of generators of $A$.
\end{proof}

A trivial example of a (right) double extension is when $A=k$.
A right double extension of $k$, denoted by $B$, is isomorphic
to $\bigoplus_{n_1,n_2\geq 0} ky_1^{n_1}y_2^{n_2}$ as a 
$k$-vector space. It is isomorphic to the algebra 
$k\langle y_1,y_2\rangle/(r)$ where $r$ is the relation 
$$y_2y_1=p_{12}y_{1}y_2+p_{11}y_1^2 +a_1 y_1+a_2y_2+a_3$$ 
for some $p_{12},p_{11},a_1,a_2,a_3\in k$.
We can easily find out the associated DE-data. The homomorphism 
$\sigma$ is the ``identity'' 
$\begin{pmatrix} Id_k&0\\0&Id_k\end{pmatrix}$;
the derivation $\delta$ is zero; the parameter $P$ is 
$\{p_{12},p_{11}\}$; the tail $\tau$ is $\{a_1,a_2,a_3\}$.
If $p_{12}\neq 0$, $B$ is a double extension. 
If $p_{12}=0$, then $B$ is not a double extension since $B$ 
is not a left double extension with respect to the generating 
set $\{y_1,y_2\}$. However, $B$ is a left double extension with 
respect to a different generating set $\{y_1',y_2'\}=
\{y_2-p_{11}y_1,y_1\}$. Some non-trivial examples are given
in Section 4.

Another way of writing (R2) is the following. 
Let $y_0=1$ and
$\sigma_{00}(r)=r$ and $\sigma_{i0}=\delta_i$
for $i=1,2$. Then (R2) is equivalent to 
\begin{equation}
\label{R2'}
\begin{pmatrix} y_0\\y_1\\y_2\end{pmatrix} r= 
\begin{pmatrix} \sigma_{00}(r)&0&0\\ 
\sigma_{10}(r)&\sigma_{11}(r)& \sigma_{12}(r)\\
\sigma_{20}(r)&\sigma_{21}(r)&\sigma_{22}(r)\end{pmatrix}
\begin{pmatrix} y_0\\ y_1\\y_2\end{pmatrix}
\quad \text{or}\quad
\begin{pmatrix} y_0\\y_1\\y_2\end{pmatrix} r=
\widehat{\sigma}(r)
\begin{pmatrix} y_0\\ y_1\\y_2\end{pmatrix}
\tag{R2'}
\end{equation}
where
$\widehat{\sigma}(r)=\begin{pmatrix}
\sigma_{00}(r)&0&0\\ 
\sigma_{10}(r)&\sigma_{11}(r)& \sigma_{12}(r)\\
\sigma_{20}(r)&\sigma_{21}(r)&\sigma_{22}(r)\end{pmatrix}$.
Similar to Lemma \ref{xxlem1.7} one can easily show
that
$\widehat{\sigma}: A\to M_3(A)$ is an algebra
homomorphism. 

The homomorphism $\sigma$ is not surjective, so it 
can not be invertible in the usually sense. Recall 
in the classical Ore extension case that 
$\sigma$ is invertible if and only if
\eqref{E1.1.1} can be written as
$$xa=\sigma '(a)x+\delta'(a)$$
for all $a\in A$; and $\sigma'$ is the inverse of
$\sigma$. We use this idea to define invertibility 
of $\sigma$ in the right double extension case. 

\begin{definition}
\label{xxdefn1.8}
Let $\sigma: A\to M_2(A)$ be an algebra homomorphism.
We say $\sigma$ is {\it invertible} if there is 
an algebra homomorphism
$\phi=\begin{pmatrix} \phi_{11}&\phi_{12}\\
\phi_{21}&\phi_{22}\end{pmatrix}: A\to M_2(A)$
satisfies the following conditions:
$$\sum_{k=1}^2 \phi_{jk}(\sigma_{ik}(r))=
\begin{cases} r& \text{if  } i=j\\
               0& \text{if  } i\neq j
\end{cases}
\quad \text{and}\quad
\sum_{k=1}^2 \sigma_{kj}(\phi_{ki}(r))=
\begin{cases} r& \text{if  } i=j\\
               0& \text{if  } i\neq j
\end{cases}
$$
for all $r\in A$; or equivalently,
$$\begin{pmatrix} \phi_{11}&\phi_{12}\\
\phi_{21}&\phi_{22}\end{pmatrix}
\bullet
\begin{pmatrix} \sigma_{11}&\sigma_{21}\\
\sigma_{12}&\sigma_{22}\end{pmatrix}
=
\begin{pmatrix} \sigma_{11}&\sigma_{21}\\
\sigma_{12}&\sigma_{22}\end{pmatrix}
\bullet
\begin{pmatrix} \phi_{11}&\phi_{12}\\
\phi_{21}&\phi_{22}\end{pmatrix}
=\begin{pmatrix} Id_A&0\\
0&Id_A\end{pmatrix}$$
where $\bullet$ is the multiplication of
the matrix algebra $M_2(\End_k(A))$. The
multiplication of $\End_k(A)$ is the 
composition of $k$-linear maps.
The map $\phi$ is called the {\it inverse} 
of $\sigma$. 
\end{definition}

\begin{lemma}
\label{xxlem1.9} Let $B=A_{P}[y_1,y_2;\sigma,\delta,\tau]$
be a right double extension. Then $\sigma$ is invertible 
if and only if $A y_1+A y_2+A=y_1 A+y_2 A+A$ and it is a free
$A$-module with basis $\{1,y_1,y_2\}$ on both sides. 

As a consequence, if $B$ is a double extension, then
$\sigma$ is invertible.
\end{lemma}

\begin{proof} Let $\phi=\begin{pmatrix} \phi_{11}& \phi_{12}\\
\phi_{21}&\phi_{22}\end{pmatrix} $ be the inverse 
of $\sigma$ defined in Definition \ref{xxdefn1.8}. Then it 
follows from (R2) and the definition of $\phi$ that,
for every $r\in A$,
$$
(y_1,y_2) \begin{pmatrix} \phi_{11}(r)& \phi_{12}(r)\\
\phi_{21}(r)&\phi_{22}(r)\end{pmatrix} = (ry_1,ry_2)-
\begin{pmatrix} \delta'_1(r),&\delta'_2(r)
\end{pmatrix}$$
where
$$\delta'_1(r)=-\delta_1(\phi_{11}(r))-\delta_2(\phi_{21}(r))$$
$$\delta'_2(r)=-\delta_1(\phi_{12}(r))-\delta_2(\phi_{22}(r)).$$
Or, after rearranging the terms, 
\begin{equation}
\label{E1.9.1}
 r (y_1,y_2)= (y_1,y_2) \begin{pmatrix} \phi_{11}(r)& \phi_{12}(r)\\
\phi_{21}(r)&\phi_{22}(r)\end{pmatrix} 
+\begin{pmatrix} \delta'_1(r),&\delta'_2(r)
\end{pmatrix}.
\tag{E1.9.1}
\end{equation}
By using \eqref{E1.9.1}, we see that $Ay_1+Ay_2+A\subseteq y_1A+y_2A+A$
and by (R2), we have $y_1A+y_2A+A\subseteq Ay_1+Ay_2+A$. Thus 
$Ay_1+Ay_2+A=y_1A+y_2A+A$. Following by the definition of right 
double extension, $Ay_1+Ay_2+A$ is a free left $A$-module with basis 
$\{1,y_1,y_2\}$, whence it is of rank 3.
It remains to show that $y_1A+y_2A+A$ is a free right $A$-module
with basis $\{1,y_1,y_2\}$. Suppose on the contrary that 
$\{1,y_1,y_2\}$ is not an $A$-basis. Then there are elements 
$\{a,b,c\}$ of $A$, not all zero, such that $y_1a+y_2b+c=0$. By 
using (R2), we have 
$$y_1 a+y_2 b+c=[\sigma_{11}(a)+\sigma_{21}(b)]y_1+
[\sigma_{12}(a)+\sigma_{22}(b)]y_2+[\delta_1(a)+\delta_2(b)+c].$$
Since $Ay_1+Ay_2+A$ is free, we have
$$\sigma_{11}(a)+\sigma_{21}(b)=\sigma_{12}(a)+\sigma_{22}(b)
=\delta_1(a)+\delta_2(b)+c=0.$$
By the definition of $\phi$, we have
$$0=\sum_{k=1}^2 \phi_{1k}(\sigma_{2k}(b))=
\sum_{k=1}^2 \phi_{1k}(-\sigma_{1k}(a))=-a$$
and
$$0=\sum_{k=1}^2 \phi_{2k}(\sigma_{1k}(a))=
\sum_{k=1}^2 \phi_{2k}(-\sigma_{2k}(b))=-b.$$
Since $a=b=0$, we also have $c=y_1a+y_2b+c=0$. This
yields a contradiction. Therefore $y_1A+y_2A+A$ is a free right 
$A$-module with basis $\{1,y_1,y_2\}$.

For the converse implication we note that every element
$ry_i$ can be expressed uniquely as $y_1\phi_{1i}(r)+y_2\phi_{2i}(r)
+\delta'_i(r)$ for some $\phi_{1i}(r),\phi_{2i}(r),\delta'_i(r)$
in $A$. Using this we can define $\phi$. One can 
check that $\phi$ is the inverse of $\sigma$.

The consequence follows from the main assertion and Remark 
\ref{xxrem1.4}(e).
\end{proof}

As a consequence of Lemma \ref{xxlem1.9} the invertibility of
$\sigma$ is independent of the choice of the DE-data 
$\{P,\sigma,\delta,\tau\}$.

We conjecture that if $\sigma$ is invertible, then $B$ is also 
a free right $A$-module with basis 
$\{y_1^{n_1}y_2^{n_2}\}_{n_1,n_2\geq 0}$. If further $p_{12}\neq 0$,
we conjecture that $B$ is a double extension.
This is true in the graded case (see Proposition \ref{xxprop1.13}).

Next we will list the relations that come from
commuting $r\in A$ with (R1). The collection of
the following six relations is called (R3) for 
short. Recall that $\sigma_{i0}=\delta_i$.

\bigskip
\bigskip

\centerline{Relations (R3)}
\begin{align}
\label{R3.1}
\sigma_{21}&(\sigma_{11}(r))+p_{11}\sigma_{22}(\sigma_{11}(r))
\tag{R3.1}\\
\notag
&=p_{11}\sigma_{11}(\sigma_{11}(r))+p_{11}^2\sigma_{12}(\sigma_{11}(r))
+p_{12}\sigma_{11}(\sigma_{21}(r))+p_{11}p_{12}\sigma_{12}(\sigma_{21}(r))
\end{align}

\begin{align}
\label{R3.2}
\sigma_{21}&(\sigma_{12}(r))+p_{12}\sigma_{22}(\sigma_{11}(r))
\tag{R3.2}\\
\notag
&=
p_{11}\sigma_{11}(\sigma_{12}(r))+p_{11}p_{12}\sigma_{12}(\sigma_{11}(r))
+p_{12}\sigma_{11}(\sigma_{22}(r))+p_{12}^2\sigma_{12}(\sigma_{21}(r))
\end{align}

\begin{align}
\label{R3.3}
\sigma_{22}&(\sigma_{12}(r))
\tag{R3.3}\\
&=p_{11}\sigma_{12}(\sigma_{12}(r))+p_{12}\sigma_{12}(\sigma_{22}(r))
\qquad\qquad\qquad\qquad\qquad\qquad\qquad\qquad\quad
\notag
\end{align}

\begin{align}
\label{R3.4}
\sigma_{20}&(\sigma_{11}(r))+\sigma_{21}(\sigma_{10}(r))+\tau_{1}
\sigma_{22}(\sigma_{11}(r))\qquad\qquad\qquad\qquad\qquad\qquad
\qquad \qquad
\tag{R3.4}\\
&=p_{11}[\sigma_{10}(\sigma_{11}(r))
+\sigma_{11}(\sigma_{10}(r))
+\tau_{1}\sigma_{12}(\sigma_{11}(r))]
\notag\\
&
\quad +p_{12}[\sigma_{10}(\sigma_{21}(t))
+\sigma_{11}(\sigma_{20}(r))
+\tau_{1}\sigma_{12}(\sigma_{21}(r))]
+\tau_{1}\sigma_{11}(r)+\tau_{2}\sigma_{21}(r)
\notag
\end{align}

\begin{align}
\label{R3.5}
\sigma_{20}&(\sigma_{12}(r))
+\sigma_{22}(\sigma_{10}(r))+\tau_{2}
\sigma_{22}(\sigma_{11}(r))
\qquad\qquad\qquad\qquad\qquad\qquad\qquad
\tag{R3.5}\\
&=p_{11}[\sigma_{10}(\sigma_{12}(r))
+\sigma_{12}(\sigma_{10}(r))
+\tau_{2}\sigma_{12}(\sigma_{11}(r))]
\notag\\
&
\quad +p_{12}[\sigma_{10}(\sigma_{22}(r))+
\sigma_{12}(\sigma_{20}(r))
+\tau_{2}\sigma_{12}(\sigma_{21}(r))]
+\tau_{1}\sigma_{12}(r)+\tau_{2}\sigma_{22}(r)
\notag
\end{align}

\begin{align}
\label{R3.6}
\sigma_{20}&(\sigma_{10}(r))+
\tau_{0}\sigma_{22}(\sigma_{11}(r))
\tag{R3.6}\\
&=p_{11}[\sigma_{10}(\sigma_{10}(r))+
\tau_{0}\sigma_{12}(\sigma_{11}(r))]
\qquad\qquad\qquad\qquad\qquad\qquad\qquad \quad \quad
\notag\\
&\quad +p_{12}[\sigma_{10}(\sigma_{20}(r))+
\tau_{0}\sigma_{12}(\sigma_{21}(r))]
+\tau_{1}\sigma_{10}(r)+\tau_{2}\sigma_{20}(r)+\tau_{0}r.
\notag
\end{align}

\bigskip

Note that the relations in (R3) are constraints between 
$P,\sigma,\delta$, and $\tau$. In other words, to form 
a right double extension, the DE-data 
$\{P,\sigma,\delta,\tau\}$ must satisfies
(R3) (as we know by Lemma \ref{xxlem1.7} that $\sigma$ 
must be an algebra homomorphism and that $\delta$ must be 
a $\sigma$-derivation). 

By using (R2) we can express an element of the form 
$y_iy_j r$ as a sum of the forms $r'y_{i'}y_{j'}$ and
$r''y_{i''}$. Combining with relation (R1), there are two 
different ways of expressing $y_2y_1 r$. One way is first 
to write $y_2y_1$ as sum of other elements by (R1), then 
to use (R2). This means we resolve $(y_2y_1)$ first in 
$y_2y_1r$. The other way is to use (R2) to move $y_1$ and 
then $y_2$ to the right-hand side of $r$, then to simplify the 
expression by using all possible relations including (R2). In 
other words, we resolve the part $(y_1 r)$ first in $y_2y_1r$. 
The process of writing $y_iy_j r$ 
as a sum of the forms $r'y_{i'}y_{j'}$ and $r''y_{i''}$
is called {\it resolving}. In general if $r,y_1,y_2$ are
elements in a ring, two different ways of resolving 
$y_2y_1 r$ could get two different expressions. 

\begin{lemma}
\label{xxlem1.10} Assume \textup{(R2') and (R1)} and
$\sigma_{i0}=\delta_i$ for $i=1,2$. 
\begin{enumerate}
\item
If we first move $y_j$ from the left-hand side of $r$ to
the right-hand side of $r$ and then move $y_i$, we have
\begin{align}\notag
y_i(y_j r)=&
[\sigma_{i1}(\sigma_{j1}(r))+p_{11}\sigma_{i2}(\sigma_{j1}(r))]y_1^2\\
\notag &+ \sigma_{i2}(\sigma_{j2}(r)) y_2^2 \\
\notag &+[\sigma_{i1}(\sigma_{j2}(r))+p_{12} 
\sigma_{i2}(\sigma_{j1}(r))]y_1y_2\\
\notag
&+[\sigma_{i0}(\sigma_{j1}(r))+\sigma_{i1}(\sigma_{j0}(r))+
\tau_{1}\sigma_{i2}(\sigma_{j1}(r))] y_1\\
\notag
&+[\sigma_{i0}(\sigma_{j2}(r))+\sigma_{i2}(\sigma_{j0}(r))+
\tau_{2}\sigma_{i2}(\sigma_{j1}(r))] y_2\\
\notag
&+[\sigma_{i0}(\sigma_{j0}(r))+\tau_{0}\sigma_{i2}(\sigma_{j1}(r))].
\end{align}
\item
Suppose $B=A_P[y_1,y_2;\sigma,\delta,\tau]$ is a right 
double  extension. 
Then six relations \textup{(R3.1)-(R3.6)} hold for
the DE-data.
\item
The six relations in \textup{(R3)} hold if and only if 
the resulting elements of $y_2y_1 r$ obtained from resolving 
$(y_2y_1)r$ and from resolving $y_2(y_1r)$ are the same.
\item
If $\delta=0$ and $\tau=0$, then the relations 
\textup{(R3.4)-(R3.6)} become trivial and the relation
\textup{(R3.1)-(R3.3)} remain unchanged. 
\end{enumerate}
\end{lemma}

\begin{proof} The proof is based on direct
computations.

(a) We use the relation \eqref{R2}:
$$y_j r=\sigma_{j1}(r)y_1+\sigma_{j2}(r) y_2
+\sigma_{j0}(r)$$
for $j=1,2$ and all $r\in A$. Then
\begin{equation}
\label{E1.10.1}
y_i(y_jr)=(y_i\sigma_{j1}(r))y_1+(y_i\sigma_{j2}(r))y_2
+
(y_i\sigma_{j0}(r)).
\tag{E1.10.1}
\end{equation}
For each $y_i\sigma_{js}(r)$ we use 
relation \eqref{R2} to have
$$y_i\sigma_{js}(r)=\sigma_{i1}(\sigma_{js}(r))y_1
+\sigma_{i2}(\sigma_{js}(r))
y_2+\sigma_{i0}(\sigma_{js}(r)).$$
Input this formula for $s=1,2$ and $0$ into equation
\eqref{E1.10.1}, 
we obtain 
\begin{align}
y_i(y_jr)&=[\sigma_{i1}(\sigma_{j1}(r))y_1
+\sigma_{i2}(\sigma_{j1}(r))
y_2+\sigma_{i0}(\sigma_{j1}(r))] y_1
\notag\\
&+[\sigma_{i1}(\sigma_{j2}(r))y_1
+\sigma_{i2}(\sigma_{j2}(r))
y_2+\sigma_{i0}(\sigma_{j2}(r))]y_2
\notag\\
&+[\sigma_{i1}(\sigma_{j0}(r))y_1
+\sigma_{i2}(\sigma_{j0}(r))
y_2+\sigma_{i0}(\sigma_{j0}(r))]
\notag
\end{align}
Then simplify this expression and use the relation
\eqref{R1} 
to obtain the desired formula.

(b,c) By associativity $(y_2y_1)r=y_2(y_1r)$. Hence
there are two ways to resolve $y_2y_1r$ into a
linear combination of lower terms if we define
$\deg y_2>\deg y_1>\deg r$ for all $r\in A$. The 
first way is to use (R1) to commute $y_2$ with 
$y_1$ then use (R2). The second way is to 
use (R2) to commute $y_1$ with $r$ first.

Starting with the relation \eqref{R1}:
$$y_2y_1=p_{12}y_1y_2+p_{11}y_1^2+\tau_{1}y_1+\tau_{2}y_{2}+\tau_{0}:=T$$
and for every $r\in A$, we have $(y_2y_1) r=Tr$ and use
the formula 
proved in part (a) to move $r$ from right to left by
``commuting'' $r$ with $y_iy_j$ and $y_i$. Hence
\begin{align}
\label{E1.10.2}
T r
&=(p_{12}y_1y_2+p_{11}y_1^2+\tau_{1}y_1+\tau_{2}y_{2}+\tau_{0})r
\tag{E1.10.2}\\
& = p_{12}(y_1y_2 r)+p_{11}(y_1^2 r)+\tau_{1}(y_1
r)+\tau_{2}(y_{2} r)+
\tau_{0} r
\notag\\
& = p_{12}\{[\sigma_{11}(\sigma_{21}(r))+p_{11}
\sigma_{12}(\sigma_{21}(r))]y_1^2
\notag\\
& \qquad\quad + \sigma_{12}(\sigma_{22}(r)) y_2^2 
\notag \\
& \qquad\quad +[\sigma_{11}(\sigma_{22}(r))+p_{12} 
\sigma_{12}(\sigma_{21}(r))]y_1y_2
\notag \\
&\qquad\quad
+[\sigma_{10}(\sigma_{21}(r))+\sigma_{11}(\sigma_{20}(r))+
\tau_{1}\sigma_{12}(\sigma_{21}(r))] y_1
\notag \\
&\qquad\quad
+[\sigma_{10}(\sigma_{22}(r))+\sigma_{12}(\sigma_{20}(r))+
\tau_{2}\sigma_{12}(\sigma_{21}(r))] y_2
\notag \\
&\qquad\quad 
+[\sigma_{10}(\sigma_{20}(r))+\tau_{0}\sigma_{12}(\sigma_{21}(r))]\}
\notag\\
&+p_{11}\{
[\sigma_{11}(\sigma_{11}(r))+p_{11}\sigma_{12}(\sigma_{11}(r))]y_1^2
\notag \\
&\qquad\quad + \sigma_{12}(\sigma_{12}(r)) y_2^2 
\notag \\
&\qquad\quad +[\sigma_{11}(\sigma_{12}(r))+p_{12} 
\sigma_{12}(\sigma_{11}(r))]y_1y_2
\notag\\
&\qquad\quad
+[\sigma_{10}(\sigma_{11}(r))+\sigma_{11}(\sigma_{10}(r))+
\tau_{1}\sigma_{12}(\sigma_{11}(r))] y_1
\notag\\
&\qquad\quad
+[\sigma_{10}(\sigma_{12}(r))+\sigma_{12}(\sigma_{10}(r))+
\tau_{2}\sigma_{12}(\sigma_{11}(r))] y_2
\notag\\ 
&\qquad\quad 
+[\sigma_{10}(\sigma_{10}(r))+\tau_{0}\sigma_{12}(\sigma_{11}(r))]\}
\notag\\
&+\tau_{1}(\sigma_{11}(r)y_1+\sigma_{12}(r)y_2+\sigma_{10}(r))
\notag\\
&+\tau_{2}(\sigma_{21}(r)y_1+\sigma_{22}(r)y_2+\sigma_{20}(r))
\notag\\
&+\tau_{0}r.
\notag
\end{align}

If using the second way by working out $y_1r$ first, we obtain
\begin{align}
\label{E1.10.3}
y_2(y_1 r)=&
[\sigma_{21}(\sigma_{11}(r))+p_{11}\sigma_{22}(\sigma_{11}(r))]y_1^2
\tag{E1.10.3}\\
\notag &+ \sigma_{22}(\sigma_{12}(r)) y_2^2 \\
\notag &+[\sigma_{21}(\sigma_{12}(r))+p_{12} 
\sigma_{22}(\sigma_{11}(r))]y_1y_2\\
\notag
&+[\sigma_{20}(\sigma_{11}(r))+\sigma_{21}(\sigma_{10}(r))+
\tau_{1}\sigma_{22}(\sigma_{11}(r))] y_1\\
\notag
&+[\sigma_{20}(\sigma_{12}(r))+\sigma_{22}(\sigma_{10}(r))+
\tau_{2}\sigma_{22}(\sigma_{11}(r))] y_2\\
\notag
&+[\sigma_{20}(\sigma_{10}(r))+\tau_{0}\sigma_{22}(\sigma_{11}(r))].
\end{align}
Comparing the coefficients of $y_1^2, y_1y_2, y_2^2,
y_1, y_2$ and $1$
respectively in equations \eqref{E1.10.2} and
\eqref{E1.10.3}, 
one sees the six relations (R3.1-R3.6). If $B$ is a right double 
extension, then $B$ is a free left $A$-module with a basis 
$\{y_1^{n_1} y_2^{n_2}\}_{n_1,n_2\geq 0}$. Thus (R3.1-R3.6) hold 
and assertion (b) follows.

If six relations (R3.1-R3.6) holds, then two
ways of resolving $y_2y_1r$ into two elements in $\sum_{n_1,n_2}
Ay_1^{n_1}y_2^{n_2}$ are the same. Thus the six relations 
(R3.1-R3.6) hold if and only if the right-hand sides of 
\eqref{E1.10.2} and \eqref{E1.10.3} are the same. Assertion
(c) follows. 

(d) Straightforward. 
\end{proof}

\begin{proposition}
\label{xxprop1.11} 
Suppose $\{P,\sigma,\delta,\tau\}$ be DE-data such that
$\sigma: A\to M_2(A)$ is a homomorphism and  
$\delta: A\to A^{\oplus 2}$ is a $\sigma$-derivation and that 
$P=\{p_{12},p_{11}\}$ and $\tau=\{\tau_{1},\tau_{2},\tau_{0}\}$. 
Assume \textup{(R3)} holds for all $r\in X$ where $X$ is a set 
of generators of $A$. Let $C$ be the algebra generated by $A$ and 
$y_1,y_2$ subject to the relations \textup{(R1)} and 
\textup{(R2)} for generators $r\in X$. Then $C$ is a right double  
extension of $A$. As a consequence, $C$ is a left free $A$-module 
with basis $\{y_1^{n_1}y_2^{n_2}\;|\; n_1,n_2\geq 0\}$.
\end{proposition}

\begin{remark}
\label{xxrem1.12}
Assume that the relations \textup{(R3)} hold for a set 
of generators of $A$. As a consequence of Proposition 
\ref{xxprop1.11}, one can form a right double extension $C$ 
and hence \textup{(R3)} holds for all $r\in A$. This 
fact is not easy to verify by a direct computation.
\end{remark}

\begin{proof}[Proof of Proposition \ref{xxprop1.11}]
Let $X:=\{x_i\}_{i\in {\mathcal O}}$ be a set of generators of
$A$ where ${\mathcal O}$ is the index set of consisting 
the first $|X|$-ordinal numbers. We write $A=k\langle X\rangle/I_A$ 
for some ideal $I_A$ of the free algebra $k\langle X\rangle$.
Further we assume that elements in $X\cup \{1\}$ are linearly 
independent. We fix a semigroup total ordering $\leq $ on the 
monomial set $\langle X\rangle$ generated by $X$ as follows. 
Let $H=x_{i_1}\cdots x_{i_m}$ and $G=x_{j_1}\cdots
x_{j_n}$ be two 
monomials. We write $H< G$ if either $m<n$ or $m=n$
and there 
is an $s\leq n$ such that $i_s<j_s$ and $i_t=j_t$ for
all $t>s$.
In particular, we have
$$1<x_1<x_2<\cdots.$$
Since ${\mathcal O}$ is well-ordered, $\leq $ satisfies
the descending chain condition.

By Bergman's diamond lemma \cite[Section 1]{Be}, there is a
reduction system $S$ associated to $I_A$,
which is a set of relations $\{f_n\}$, such that all
ambiguity between $f_n$'s can 
be resolved. Each relation $f_n$ is of the form
$$f^l_n-H_n=0$$
where the leading term $f^l_n$ is a monomial and $H_n$
is a $k$-linear combination of monomials of 
lower order and not appeared as the
leading term of other relations in the reduction
system $S$. Let $Basis_A$ be the set of monomials that 
do not contain $f^l_n$ as a submonomial for all 
$f_n\in S$ (in \cite{Be} $Basis_A$ is denoted by 
$k\langle X\rangle_{irr}$). By \cite[Theorem 1.2]{Be}, 
$Basis_A$ is a $k$-linear basis of $A$. 

Now we define a total ordering $\leq$ on the set of 
monomials $\langle X\cup\{y_1,y_2\} \rangle$ 
generated by the set of generators 
$X\cup \{y_1,y_2\}$ of $C$. 
Let $I=F_m y_{i_m}F_{m-1} y_{i_{m-1}}\cdots
F_1 y_{i_1} F_{0}$
and $J=G_n y_{j_n}G_{n-1} y_{j_{n-1}}\cdots G_1
y_{j_1} G_{0}$ 
be two monomials where $F_s$ and $G_t$ are 
monomials in $\langle X\rangle$ for all $s$ and 
$t$. We write $I<J$ if either 
\begin{enumerate}
\item
$m<n$, or 
\item
$m=n$ and there is an $s$ such that
$F_t=G_t$ and 
$y_{i_t}=y_{j_t}$ for all $t<s$ and $y_{i_s}<y_{j_s}$,
or 
\item
$m=n$ and there is an $s$ such that $F_{t-1}=G_{t-1}$
and 
$y_{i_t}=y_{j_t}$ for all $t<s$ and $F_{s-1}<G_{s-1}$.
\end{enumerate}
This defines a total ordering on the
monomial set $\langle X\cup\{y_1,y_2\}\rangle$ which extends 
the ordering on $\langle X\rangle$; and this ordering is 
preserved by the semigroup multiplication and satisfies 
the descending chain condition. Let $T$ be the 
reduction system of $k\langle X\cup\{y_1,y_2\}\rangle$ 
consisting $S$, (R1) and (R2)(X) where (R2)(X) denotes the 
set of relations of form (R2) for all $r$ in the generating 
set $X$. By definition, every element in $T$ is a relation 
of the algebra $C$ when identifying $C=k\langle 
X\cup\{y_1,y_2\}\rangle/I_C$. By Lemma \ref{xxlem1.7}(c), (R2) 
holds for all $r\in A$ if and only if (R2)(X) holds. We claim 
that all ambiguities of $T$ are resolvable in the sense of 
\cite{Be}. All ambiguities of $S$ are resolvable by the choice 
of $S$. So we only need to consider ambiguities of the relations 
in $T$ which involving $y_i$. First we consider an ambiguity
involving (R2)(X) and $S$. Namely, we pick a relation from 
(R2)(X), say 
$$y_i x_j=\sigma_{i1}(x_j) y_1+\sigma_{i2}(x_j)
y_2+\sigma_{i0}(x_j)$$
for $i$ being either $1$ or $2$ and for some $j$; 
and pick another relation from $S$, say
$$F=G$$ 
where $F=x_j f$ and where $G$ is a linear 
combination of monomials lower than $F$. We want 
to show the following: two different ways of
reducing $y_iF$ (starting from $(y_ix_j)f$
or from $y_i (x_j f)$)
have the same result. We need to use the fact
$\widehat{\sigma}$ is an algebra
homomorphism and induction on the order of $F$. Again
let $F=x_j f$.
Since the degree of $f$ is less than the degree of
$F$, by induction on the degree of $F$,
there is a unique way of expressing $y_s f$ that is
$$y_s
f=\sigma_{s1}(f)y_1+\sigma_{s2}(f)y_2+\sigma_{s0}(f).$$
So the first way of reducing $y_iF$ is 
\begin{align} 
(y_i x_j)f&=(\sigma_{i1}(x_j) y_1+\sigma_{i2}(x_j)
y_2+\sigma_{i0}(x_j))f
\notag
\\
&=\sum_{s=0}^2
\sigma_{is}(x_j)(\sum_{k=0}^2\sigma_{sk}(f)y_k)
\notag\\
&=\sum_{k,s=0}^2\sigma_{is}(x_j)\sigma_{sk}(f) y_k
\notag\\
&=\sum_{k=0}^2 \sigma_{ik}(x_j f) y_k
\notag\\
&=\sum_{k=0}^2 \sigma_{ik}(F) y_k.
\notag
\end{align}
The second way of reducing $y_iF$ is 
$$y_iF=y_i G=\sigma_{i1}(G) y_1+\sigma_{i2}(G)
y_2+\sigma_{i0}(G)
=\sum_{k=0}^2 \sigma_{ik}(F) y_k$$
where the first equality follows from $F=G$ where
$G$ is a linear combination of lower monomials than $F$
and where the last equality holds because
$\widehat{\sigma}(F)=
\widehat{\sigma}(G)$ in $A$. So we proved our claim.
In particular, the ambiguity of $y_ix_j f$ can be resolved. 
There is another kind of ambiguities we need to resolve, namely,
$y_2y_1x_j$ for all $j$. By hypothesis, (R3)(X) holds where 
(R3)(X) denotes the set of relations of the form (R3.1-6) when 
applying to $r$ in the generating set $X$. By Lemma \ref{xxlem1.10}(c),
these ambiguities can be resolved. So we have proved that 
all ambiguities between elements in $T$ can be resolved.

Let $Basis_C$ be the set of monomials that do not contain any 
submonomial that is a leading term of elements in $T$. By Bergman's 
diamond lemma \cite[Theorem 1.2]{Be}, $Basis_C$ is a $k$-linear 
basis of $C$. 
The main assertion in Proposition \ref{xxprop1.11} follows if the 
following claim holds: $Basis_C$ is equal to the set 
$B'=\bigcup_{n_1,n_2\geq 0} Basis_A \cdot y_1^{n_1}y_2^{n_2}$. 
We show this in the next paragraph.

Recall that $T$ consists of (R1), (R2)(X) and $S$ where $S$ is a 
reduction system of $A$. Let $f\in Basis_A$, namely, $f$ does not 
contain any submonomial that is a leading term of the elements 
in $S$. Since the leading terms of (R1) and (R2)(X) are $y_2y_1$ and
$y_ix_j$, $f y_1^{n_1}y_2^{n_2}$ does not contain any submonomial that
is a leading term of the elements in $T$. This shows that $B'$
is a subset of $Basis_C$. Conversely, let $F =F_my_{i_{m}}F_{m-1}
\cdots F_1y_{i_1}$ be a monomial that does not contain 
any submonomial that is a leading term of $T$ where $F_i\in 
\langle X\rangle$. Since $F$ does not contain $y_ix_j$ (see
the relation (R2)(X)), all $F_i=1$ for all $i<m$. Thus 
$F=F_m y_{1}^{s_1}y_{2}^{t_1} \cdots y_{1}^{s_p}y_{2}^{t_p}$. Since 
$F$ does not contain $y_2y_1$ (see the relation (R1)), $F=F_m y_1^sy_2^t$.
Finally $F_m\in Basis_A$. Therefore the claim is proved.
\end{proof}

In some cases a right double extension is automatically 
a left double extension. Here is an example. 

\begin{proposition}
\label{xxprop1.13} Let $A$ be a connected graded algebra
and $B=A_{P}[y_1,y_2;\sigma,\delta,\tau]$ be a 
connected graded right double extension. Suppose that 
$p_{12}\neq 0$ and $\sigma$ is invertible. Then $B$
is a double extension.
\end{proposition}

\begin{proof} Since $\sigma$ is invertible with inverse
$\phi$, we have \eqref{E1.9.1}
$$
 r (y_1,y_2)= (y_1,y_2) \begin{pmatrix} \phi_{11}(r)& \phi_{12}(r)\\
\phi_{21}(r)&\phi_{22}(r)\end{pmatrix} 
+\begin{pmatrix} \delta'_1(r)\\ \delta'_2(r)
\end{pmatrix}
$$
for some $\delta'_1$ and $\delta'_2$.
\eqref{E1.9.1} implies that $Ay_1+Ay_2\subset y_1A+y_2A+A$.
Hence Definition \ref{xxdefn1.3}(biv) holds.

Since $p_{12}\neq 0$, (R1) is equivalent to
\begin{equation}
\label{E1.13.1}
y_1y_2=p_{12}^{-1}y_2y_1+(-p_{12}p_{11})y_1^2+
y_1\tau_{1}'+y_2\tau_{2}'+\tau_{0}'
\tag{E1.13.1}
\end{equation}
for some $\{\tau_{1}',\tau_{2}',\tau_{0}'\}$. 
Hence Definition \ref{xxdefn1.3}(bii) (namely, \eqref{L1}) 
holds. Definition \ref{xxdefn1.3}(bi) is automatic. 

It remains to show Definition \ref{xxdefn1.3}(biii), namely, 
$B$ is a right free $A$-module with 
a basis $\{y_2^{n_1}y_{1}^{n_2}\}_{n_1,n_2\geq 0}$.
Since $B$ is a right double extension, it is a free
left $A$-module with a basis $\{y_1^{n_1}y_2^{n_2}\}_{n_1,n_2\geq 0}$. Hence
the Hilbert series of $B$ is equal to the Hilbert series of
$A\otimes k[y_1,y_2]$, or
\begin{equation}
\label{E1.13.2}
H_B(t)=H_{A}(t) {\frac{1}{(1-t^{\deg y_1})(1-t^{\deg y_2})}}.
\tag{E1.13.2}
\end{equation}
By using \eqref{E1.9.1} and \eqref{E1.13.1} we have $B=\sum_{n_1,n_2}
y_2^{n_1}y_1^{n_2} A$. Since the Hilbert series of $B$ is 
equal to the Hilbert series of $k[y_1,y_2]\otimes A$,
$B$ is a free right $A$-module with basis 
$\{y_2^{n_1}y_1^{n_2}\}_{n_1,n_2\geq 0}$.
\end{proof}

We hope that Proposition \ref{xxprop1.13} holds
for general non-graded double extensions, but this
has not been proven yet.

Finally we like to remark that the condition $p_{12}\neq 0$ is
natural if we study noetherian algebras. 

\begin{proposition}
\label{xxprop1.14}
Let $B=A_{P}[y_1,y_2;\sigma,\delta,\tau]$ be a connected graded 
right double extension such that either $\deg y_1=\deg y_2>0$
or $\sigma(A_{\geq 1})\subset M_2(A_{\geq 1})$.
\begin{enumerate}
\item
$B/(A_{\geq 1})$ is isomorphic to $k\langle y_1,y_2\rangle/
(y_2y_1-p_{12}y_1y_2-p_{11}y_1^2)$.
\item
If $B$ is left or right noetherian, then $p_{12}\neq 0$.
\end{enumerate}
\end{proposition}

\begin{proof} If $\deg y_1=\deg y_2>0$, then by (R2) 
$\sigma(A_{\geq 1})\subset M_2(A_{\geq 1})$. For the rest
of the proof we assume that $\sigma(A_{\geq 1})\subset 
M_2(A_{\geq 1})$.

(a) By definition, $B=\bigoplus_{n_1,n_2\geq 0} 
Ay_1^{n_1}y_2^{n_2}$ is a free left $A$-module. Let $I$
be the left $A$-module $\bigoplus_{n_1,n_2\geq 0} 
A_{\geq 1}y_1^{n_1}y_2^{n_2}$, which is equal to the right 
ideal $A_{\geq 1}B$. Then $B/I$ is isomorphic to 
$\bigoplus_{n_1,n_2\geq 0} ky_1^{n_1}y_2^{n_2}$ as graded 
vector space. We claim that $I$ is a 
two-sided ideal of $B$. Clearly $AI\subset I$. By (R2) and 
the hypothesis of $\sigma(A_{\geq 1})\subset M_2(A_{\geq 1})$,
$y_iA_{\geq 1}\subset A_{\geq 1}y_1+A_{\geq 1}y_2+A_{\geq 1}$.
Since $B$ is generated by $A$ and $y_1,y_2$,  $I$
is a left and hence two-sided ideal of $B$. Further 
$I$ is the two-sided ideal of $B$ generated by $A_{\geq 1}$.
The factor ring $B/I$ is generated by $y_1$ and $y_2$ subject 
to one relation $y_2y_1=p_{12}y_1y_2+p_{11}y_1^2$. 

(b) If $p_{12}=0$, then $B/I$ is isomorphic to $k\langle 
y_1,y_2\rangle/((y_2-p_{11}y_1)y_1)$ which is neither left nor 
right noetherian. This yields a contradiction because
$B$ is left or right noetherian. Therefore $p_{12}\neq 0$.
\end{proof}

The condition $p_{12}\neq 0$ is also related to the Artin-Schelter 
regularity, see \cite[Section 2]{ZZ}. We will further discuss 
some homological properties of double extensions in Section 
\ref{xxsec3}.

\section{Determinant of $\sigma$}
\label{xxsec2}

In this section we study an invariant of $\sigma$, called the 
determinant of $\sigma$, which is somewhat related to the determinant 
of the quantum $2\times 2$ matrix. The property of the determinant
of $\sigma$ will be used to prove the regularity of double extensions
in the next section.

As before we write $\sigma$ as a matrix form 
$\begin{pmatrix} \sigma_{11}&\sigma_{12}\\
                 \sigma_{21}&\sigma_{22}
\end{pmatrix}$. It is natural to ask if there is
any kind of determinant of this ``matrix'' $\sigma$. Suppose $B=A_{P}[y_1,y_2;
\sigma,\delta,\tau]$ is a right double extension. We define
the {\it $P$-determinant} of $\sigma$ or just {\it determinant} 
of $\sigma$ to be the map 
$$\det \sigma: r\mapsto -p_{11}\sigma_{12}(\sigma_{11}(r))
+\sigma_{22}(\sigma_{11}(r))-p_{12}\sigma_{12}(\sigma_{21}(r))$$
for all $r\in A$. This is a $k$-linear map from $A$ to itself.
In the classical case of $P=(1,0)$, we can write 
$$\det \sigma=\sigma_{22}\circ \sigma_{11}-\sigma_{12}\circ
\sigma_{21}.$$
If $p_{12}\neq 0$, then by (R3.2) we obtain that
$$\det \sigma=-p_{12}^{-1}p_{11}\sigma_{11}\circ\sigma_{12}
 -p_{12}^{-1}\sigma_{21}\circ\sigma_{12}
+\sigma_{11}\circ\sigma_{22},$$
which implies that, if $P=(1,0)$,
$$\det \sigma=\sigma_{11}\circ\sigma_{22}-\sigma_{21}\circ
\sigma_{12}.$$
However, since $f\circ g\neq g\circ f$ generally in 
$\End_k(A)$, when $P=(p,0)$,
$$\det\sigma\neq \sigma_{22}\circ\sigma_{11}-p\sigma_{21}\circ
\sigma_{12},\quad \det\sigma
\neq \sigma_{11}\circ \sigma_{22}-p^{-1}\sigma_{12}\circ
\sigma_{21}.$$
See Proposition \ref{xxprop4.6}(c) for an example. 

\begin{proposition}
\label{xxprop2.1} Let $A_{P}[y_1,y_2;\sigma,\delta,\tau]$ be
a right double extension.
\begin{enumerate}
\item
$\det \sigma$ is an algebra endomorphism of $A$.
\item
If $\sigma$ is invertible, then $\det \sigma$ is invertible.
\item
Suppose $p_{12}\neq 0$. If $\det \sigma$ is invertible,
then $\sigma$ has a right inverse. If further $B$ is a 
connected graded right double extension, then $\sigma$ 
is invertible and $B$ is a double extension. 
\end{enumerate}
\end{proposition}

\begin{proof} First of all we may assume that $\delta=0$ and 
$\tau=0$ since the assertions are not related to $\delta$
and $\tau$. Secondly, by Remark \ref{xxrem1.4}(b), we may 
further assume that $P$ is either $(0,0)$, or $(p,0)$, 
or $(1,1)$ (which is in fact not essential). It turns out 
that the the proof is easiest when 
$P=(0,0)$ and that the proof is most tedious when $P=(1,1)$. 
So we only give all details in the intermediate case when 
$P=(p,0)$ and $p\neq 0$. In this case relations
(R3.1), (R3.2) and (R3.3) become
\begin{equation}
\label{E2.1.1}
\sigma_{21}\circ \sigma_{11}=p\sigma_{11}\circ \sigma_{21}\tag{E2.1.1}
\end{equation}
\begin{equation}
\label{E2.1.2}
\sigma_{21}\circ \sigma_{12}=-p\sigma_{22}\circ \sigma_{11}
+p\sigma_{11}\circ \sigma_{22}+p^2\sigma_{12}\circ\sigma_{21}\tag{E2.1.2}
\end{equation}
\begin{equation}
\label{E2.1.3}
\sigma_{22}\circ \sigma_{12}=p\sigma_{12}\circ \sigma_{22}.\tag{E2.1.3}
\end{equation}

(a) The following computation uses (E2.1.1)-(E2.1.3).
$$
\begin{aligned}
\det \sigma(r\cdot a)&=[\sigma_{11}\sigma_{22}-p^{-1}
\sigma_{21}\sigma_{12}](ra)\\
&=
\sigma_{11}[\sigma_{21}(r)\sigma_{12}(a)+\sigma_{22}(r)\sigma_{22}(a)]\\
&\quad 
-p^{-1}\sigma_{21}[\sigma_{11}(r)\sigma_{12}(a)+\sigma_{12}(r)\sigma_{22}(a)]\\
&=\sigma_{11}\sigma_{21}(r)\sigma_{11}\sigma_{12}(a)
 +\sigma_{12}\sigma_{21}(r)\sigma_{21}\sigma_{12}(a)\\
&\quad
+\sigma_{11}\sigma_{22}(r)\sigma_{11}\sigma_{22}(a)
+\sigma_{12}\sigma_{22}(r)\sigma_{21}\sigma_{22}(a)\\
&\quad
-p^{-1}\sigma_{21}\sigma_{11}(r)\sigma_{11}\sigma_{12}(a)
-p^{-1}\sigma_{22}\sigma_{11}(r)\sigma_{21}\sigma_{12}(a)\\
&\quad
-p^{-1}\sigma_{21}\sigma_{12}(r)\sigma_{11}\sigma_{22}(a)
-p^{-1}\sigma_{22}\sigma_{12}(r)\sigma_{21}\sigma_{22}(a)\\
&=\sigma_{12}\sigma_{21}(r)\sigma_{21}\sigma_{12}(a)\\
&\quad
+\sigma_{11}\sigma_{22}(r)\sigma_{11}\sigma_{22}(a)\\
&\quad
-p^{-1}\sigma_{22}\sigma_{11}(r)\sigma_{21}\sigma_{12}(a)\\
&\quad
-p^{-1}\sigma_{21}\sigma_{12}(r)\sigma_{11}\sigma_{22}(a)
\qquad\qquad\qquad\quad
{\text{by (E2.1.1) and (E2.1.3)}}\\
&=\sigma_{11}\sigma_{22}(r)\sigma_{11}\sigma_{22}(a)
-p^{-1}\sigma_{21}\sigma_{12}(r)\sigma_{11}\sigma_{22}(a)\\
&\quad
+\sigma_{12}\sigma_{21}(r)\sigma_{21}\sigma_{12}(a)
-p^{-1}\sigma_{22}\sigma_{11}(r)\sigma_{21}\sigma_{12}(a)\\
&=\sigma_{11}\sigma_{22}(r)\sigma_{11}\sigma_{22}(a)
-p^{-1}\sigma_{21}\sigma_{12}(r)\sigma_{11}\sigma_{22}(a)\\
&\quad
+p^{-2}\sigma_{21}\sigma_{12}(r)\sigma_{21}\sigma_{12}(a)
-p^{-1}\sigma_{11}\sigma_{22}(r)\sigma_{21}\sigma_{12}(a)
\qquad
{\text{by (E2.1.2)}}\\
&=[\sigma_{11}\sigma_{22}(r)-p^{-1}\sigma_{21}\sigma_{12}(r)]
[\sigma_{11}\sigma_{22}(a)-p^{-1}\sigma_{21}\sigma_{12}(a)]\\
&=\det \sigma(r)\det\sigma(a).
\end{aligned}
$$

(b) Let $\phi$ be the inverse of $\sigma$ as in the Definition
\ref{xxdefn1.8}. For this part we also need the relations (R3) 
for $\phi$, which are listed below. (These can be obtained by
going through the work in Section 1 for left double extensions).

\begin{equation}
\label{E2.1.4}
\phi_{11}\circ \phi_{12}=p\phi_{12}\circ \phi_{11}\tag{E2.1.4}
\end{equation}
\begin{equation}
\label{E2.1.5}
\phi_{21}\circ \phi_{12}=-p\phi_{11}\circ \phi_{22}
+p\phi_{22}\circ \phi_{11}+p^2\phi_{12}\circ\phi_{21}\tag{E2.1.5}
\end{equation}
\begin{equation}
\label{E2.1.6}
\phi_{21}\circ \phi_{22}=p\phi_{22}\circ \phi_{12}.\tag{E2.1.6}
\end{equation}

The determinant of $\phi$ is
$$\det \phi=\phi_{11}\circ\phi_{22}-p\phi_{12}\circ\phi_{21}=
\phi_{22}\circ\phi_{11}-p^{-1}\phi_{21}\circ\phi_{12}.$$
We claim that 
$$\det \phi\circ\det \sigma=\det \phi\circ\det \sigma=Id_A.$$
We only prove that $\det \sigma\circ \det \phi=Id_A$ and 
the proof of $\det \phi\circ\det \sigma=Id_A$ is similar.

$$
\begin{aligned}
\det\sigma\det \phi
&=(\sigma_{22}\sigma_{11}-p\sigma_{12}\sigma_{21})\det \phi
=\sigma_{22}\sigma_{11}\det \phi-p\sigma_{12}\sigma_{21}\det \phi\\
&=\sigma_{22}\sigma_{11}(\phi_{11}\phi_{22}-p\phi_{12}\phi_{21})
-p\sigma_{12}\sigma_{21}(\phi_{22}\phi_{11}-p^{-1}\phi_{21}\phi_{12})\\
&=\sigma_{22}\sigma_{11}\phi_{11}\phi_{22}
-p\sigma_{22}\sigma_{11}\phi_{12}\phi_{21}\\
&\quad -p\sigma_{12}\sigma_{21}\phi_{22}\phi_{11}
+\sigma_{12}\sigma_{21}\phi_{21}\phi_{12}\\
&=\sigma_{22}[Id_A-\sigma_{21}\phi_{21}]\phi_{22}
+p\sigma_{22}\sigma_{21}\phi_{22}\phi_{21}\\
&\quad
p\sigma_{12}\sigma_{11}\phi_{12}\phi_{11}+
\sigma_{12}[Id_A-\sigma_{11}\phi_{11}]\phi_{12}
\qquad\qquad\quad
{\text{by Definition 1.8}}\\
&=\sigma_{22}\phi_{22}-\sigma_{22}\sigma_{21}\phi_{21}\phi_{22}
+\sigma_{22}\sigma_{21}\phi_{21}\phi_{22}\\
&\quad
\sigma_{12}\sigma_{11}\phi_{11}\phi_{12}+
\sigma_{12}\phi_{12}-\sigma_{12}\sigma_{11}\phi_{11}\phi_{12}
\qquad
{\text{by (E2.1.4) and (E2.1.6)}}\\
&=\sigma_{22}\phi_{22}+\sigma_{12}\phi_{12}\\
&=Id_A.\qquad\qquad\quad\qquad\qquad\quad
\qquad\qquad\quad\qquad\qquad\;
{\text{by Definition 1.8}}
\end{aligned}
$$

The assertion follows.

(c) Let $d=\det \sigma$ and let 
$$\phi=\begin{pmatrix} \phi_{11}&\phi_{12}\\
                       \phi_{21}&\phi_{22}
\end{pmatrix}
=\begin{pmatrix} \sigma_{22}\circ d^{-1}&-p\sigma_{21}\circ d^{-1}\\
                 -p^{-1}\sigma_{12}\circ d^{-1}&\sigma_{11}\circ d^{-1}
\end{pmatrix}.$$
It is straightforward to prove that $\sigma\bullet \phi=Id_{2\times 2}$
where $\bullet$ is defined in Definition \ref{xxdefn1.8}.
Hence $\phi$ is a right inverse of $\sigma$. 

We now assume that $B$ is a connected graded right double extension.
By the relation (R1), $y_1 A+y_2A+A\subseteq Ay_1+Ay_2+A$. Since 
$\phi$ is a right inverse of $\sigma$ we also obtain \eqref{E1.9.1}:
$$r(y_1,y_2)=(y_1,y_2)\begin{pmatrix} \phi_{11}(r)&\phi_{12}(r)\\
                       \phi_{21}(r)&\phi_{22}(r)
\end{pmatrix}+(\delta'_1(r),\delta'_2(r))$$
for all $r\in A$. This implies that $Ay_1+Ay_2+A\subseteq
y_1A+y_2A+A$ and hence $Ay_1+Ay_2+A=y_1A+y_2A+A$.
Since $B$ is a right double extension,
$Ay_1+Ay_2+A$ is a left free $A$-module with basis $\{1,y_1,y_2\}$. In 
particular the Hilbert series of $Ay_1+Ay_2+A$ is equal to
$(1+t^{\deg y_1}+t^{\deg y_2})H_A(t)$. Therefore the Hilbert series of
$y_1A+y_2A+A$ is also $(1+t^{\deg y_1}+t^{\deg y_2})H_A(t)$. This implies
that $y_1A+y_2A+A$ is a right free $A$-module with basis $\{1,y_1,y_2\}$.
By Lemma \ref{xxlem1.9}, $\sigma$ is invertible.

By Proposition \ref{xxprop1.13}, $B$ is a double extension.
\end{proof}

Next we prove that the determinant $\det \sigma$ has a homological 
interpretation. If $M$ is a $(B,A)$-bimodule and $f$ is an automorphism
of $A$, then the twisted bimodule ${^1M^f}=M^f$ is defined to be
$M$ as $k$-space with bimodule action 
$$b\cdot m\cdot a=bmf(a)$$
for all $m\in M,b\in B,a\in A$. Similarly for an $(A,B)$-bimodule
$N$, twisted $(A,B)$-bimodule ${^f} N$ can be defined. 

For $M\oplus M$ there is another
right $A$-module structure defined by using $\sigma$ as follows:
$$(s,t)*r=(s,t)\begin{pmatrix} \sigma_{11}(r)&\sigma_{12}(r)\\
                     \sigma_{21}(r)&\sigma_{22}(r)
\end{pmatrix}
=(s\sigma_{11}(r)+t\sigma_{21}(r),
               s\sigma_{12}(r)+t\sigma_{22}(r))$$
for all $s,t\in M$ and $r\in A$. Since $\sigma$ is an algebra
homomorphism, $*$ defines a right $A$-module structure
commuting with the left $B$-module structure.
This $(B.A)$-bimodule is denoted by $(M\oplus M)^\sigma$.

\begin{theorem} 
\label{xxthm2.2}
Let $B$ be a trimmed double extension $A_{P}[y_1,y_2;\sigma]$.
\begin{enumerate}
\item
There is an exact sequence of $(B,A)$-modules
\begin{equation}
\label{E2.2.1}
0\to B^{\det \sigma} \to  (B\oplus B)^{\sigma}
\to B\to A\to 0
\tag{E2.2.1}
\end{equation}
where the last term $A$ is identified with $B/(y_1,y_2)$.
\item
$$\Ext^i_B(A,B)\cong
\begin{cases} {^{\det \sigma}A} & i=2\\
0& i\neq 2\end{cases}.$$
\item
If $B$ is a connected graded double extension, then 
the graded versions of (a) and (b) hold.
\end{enumerate}
\end{theorem}

\begin{proof} (a) It follows from Definition 
\ref{xxdefn1.3}(aii,aiii) that we have an exact 
sequence of left $B$-modules
\begin{equation}
\label{E2.2.2}
0\to B \xrightarrow{g} B\oplus B
\xrightarrow{f} B\xrightarrow{\epsilon} A\to 0
\tag{E2.2.2}
\end{equation}
where $f$ maps $(a,b)$ to $(a,b)\begin{pmatrix}y_1\\
y_2\end{pmatrix}=ay_1+by_2$ and
$g$ maps $c$ to $(c(p_{11}y_1-y_2),cp_{12}y_1)$.
Since $B$ contains $A$, $B$ has a natural right $A$-module
structure and $\epsilon$ is a $(B,A)$-bimodule map.
Let $*$ be the right $A$-module on $(B\oplus B)^\sigma$
defined before the theorem. Then 
$$(a,b)*r
=(a\sigma_{11}(r)+b\sigma_{21}(r),a\sigma_{12}(r)+b\sigma_{22}(r)).
$$
Since $\sigma$ is an algebra homomorphism, $*$ defines a right 
$A$-module structure. Since 
$$f((a,b)*r)=f((a,b)\sigma(r))=(a,b)\sigma(r)\begin{pmatrix}y_1\\
y_2\end{pmatrix}=(a,b)\begin{pmatrix}y_1\\
y_2\end{pmatrix} r=f(a,b)r,$$
$f$ is a right $A$-module map. Let $\star$ be the right $A$-module
on $B^{\det \sigma}$, namely, 
$$c\star r=c\det\sigma(r)$$
for all $c\in B$ and $r\in A$. Since $\det \sigma$ is invertible
(Proposition \ref{xxprop2.1}(b)) and  since $A^{\det \sigma}\cong A$ 
as right $A$-module, we see that $B^{\det \sigma}$ is isomorphic
to a free right $A$-module $\bigoplus_{n_1,n_2\geq 0} 
y_1^{n_1}y_2^{n_2} A$. Next we verify that $g$ is a right $A$-module 
map. For $c\in B$ and $r\in A$, 
$$
\begin{aligned}
g(c)*r&= (c(p_{11}y_1-y_2),cp_{12}y_1)*r\\
&=(c(p_{11}y_1-y_2)\sigma_{11}(r)+cp_{12}y_1\sigma_{21}(r),
   c(p_{11}y_1-y_2)\sigma_{12}(r)+cp_{12}y_1\sigma_{22}(r)).
\end{aligned}
$$
We need to change this expression by moving $\sigma_{ij}(r)$
form the right-hand side of $y_s$ to the left-hand side of $y_t$.
Since $g$ is a left $B$-module map, we may further assume that
$c$ is $1$. The first component is now
$$
\begin{aligned}
(p_{11}y_1-y_2)&\sigma_{11}(r)+p_{12}y_1\sigma_{21}(r)\\
&=
p_{11}\sigma_{11}\sigma_{11}(r)y_1+p_{11}\sigma_{12}\sigma_{11}(r)y_2\\
&\quad
-\sigma_{21}\sigma_{11}(r)y_1-\sigma_{22}\sigma_{11}(r)y_2\\
&\quad
+p_{12}\sigma_{11}\sigma_{21}(r)y_1+p_{12}\sigma_{12}\sigma_{21}(r)y_2\\
&=(p_{11}\sigma_{11}\sigma_{11}(r)-\sigma_{21}\sigma_{11}(r)
+p_{12}\sigma_{11}\sigma_{21}(r))y_1-(\det\sigma)(r)y_2\\
&=p_{11}(\sigma_{22}\sigma_{11}-p_{11}\sigma_{12}\sigma_{11}(r)-
p_{12}\sigma_{12}\sigma_{21}(r))y_1-(\det\sigma)(r)y_2 \\
&\qquad\qquad\qquad\qquad\qquad\qquad\qquad\qquad\qquad
\qquad\qquad\qquad
{\text{by (R3.1)}}\\
&=(\det \sigma)(p_{11}y_1-y_2).
\end{aligned}
$$
The second component is now
$$
\begin{aligned}
(p_{11}y_1-y_2)&\sigma_{12}(r)+p_{12}y_1\sigma_{22}(r)\\
&=p_{11}\sigma_{11}\sigma_{12}(r)y_1+p_{11}\sigma_{12}\sigma_{12}(r)y_2\\
&\quad
-\sigma_{21}\sigma_{12}(r)y_1-\sigma_{22}\sigma_{12}(r)y_2\\
&\quad
+p_{12}\sigma_{11}\sigma_{22}(r)y_1+p_{12}\sigma_{12}\sigma_{22}(r)y_2\\
&=(p_{11}\sigma_{11}\sigma_{12}(r)-\sigma_{21}\sigma_{12}(r)
+p_{12}\sigma_{11}\sigma_{22}(r))y_1 
\qquad\qquad {\text{by (R3.3)}}\\
&=(p_{12}^{-1}p_{11}\sigma_{11}\sigma_{12}(r)
-p_{12}^{-1}\sigma_{21}\sigma_{12}(r)
+\sigma_{11}\sigma_{22}(r))p_{12}y_1 \\
&=(\det\sigma)(r) p_{12}y_1.
\end{aligned}
$$
Thus 
$$g(c)* r=(c(\det\sigma)(r)(p_{11}y_1-y_2),c(\det\sigma)(r)p_{12}y_1)
=g(c(\det\sigma)(r))=g(c\star r).$$
Hence $g$ is a $(B,A)$-bimodule map and we have an exact sequence 
of $(B,A)$-bimodules \eqref{E2.2.1}.

(b) To compute $\Ext^2_B(A,B)$, we apply the functor $(-)^\vee:=
\Hom_B(-,B)$ to the truncated sequence of \eqref{E2.2.1} and 
obtain a complex of $(A,B)$-modules
$$0\leftarrow {^{\det \sigma}B }\xleftarrow{g^\vee} {^{\sigma}(B\oplus B)}
\xleftarrow{f^\vee} B\leftarrow 0.$$
Since $B$ is a left double extension of $A$, the above
complex is exact except for the homology at ${^{\det \sigma}B }$
position. Hence $\Ext^i_B(A,B)=0$ for all $i\neq 2$ and
$$\Ext^2_{B}(A,B)={^{\det \sigma}B }/im (g^\vee)\cong 
{^{\det \sigma}B }/(y_1,y_2)={^{\det \sigma}A }.$$

(c) This is clear since all maps can be chosen to be graded.
\end{proof}

\section{Regularity of double extensions}
\label{xxsec3}

In this section we recall the definition of Artin-Schelter
regularity and prove Theorem \ref{xxthm0.2}.

The right derived functor
of $\Hom_B(-,-)$ is denoted by $\RHom_B$
and the left derived functor of $\otimes_B$
is denoted by $\otimes^L_B$. In particular,
$$H^i(\RHom_B(M,N))=\Ext^i_B(M,N)$$
where $H^i$ is the $i$th cohomology of a complex, and 
$$H^{-i}(M\otimes^L_B N)=\Tor^B_i(M,N).$$
When $B$ is graded, we also take these functors and
derived functors in the graded module category. 
Let $X$ be a complex. Then $X[d]$ means the $d$th
complex shift and $X(d)$ means $d$th degree shift
where the degree shift comes from the grading of
$B$. 

A connected graded algebra $B=k\oplus B_1\oplus B_2\oplus \cdots$ 
is called {\it Artin-Schelter regular} or {\it regular} for short 
if the following three conditions hold.
 \begin{enumerate}
 \item[(AS1)] $B$ has finite global dimension $d$, and
 \item[(AS2)] $B$ is {\it Gorenstein}, namely, there is an
integer $l$ such that,
 $$\Ext^i_B({_Bk}, B)=\begin{cases} k(l) & \text{ if }
i=d\\
                                0   & \text{ if }
i\neq d
 \end{cases}
 $$
 where $k$ is the trivial $B$-module $B/B_{\geq 1}$; and
the same condition holds for the right trivial $B$-module
$k_B$.
\item[(AS3)] 
$B$ has finite Gelfand-Kirillov dimension, i.e., there
is a 
positive number $c$ such that $\dim B_n< c\; n^c$ for
all 
$n\in \mathbb{N}$.
 \end{enumerate}

If $B$ is regular, then the global dimension of $B$ is called the
{\it dimension} of $B$. Condition (AS2) means that 
$\RHom_B(k,B)=k(l)[d]$.

If $B$ is regular, then by \cite[3.1.1]{SZ}, the trivial left 
$B$-module $_Bk$ has a minimal free resolution of the form
\begin{equation}
\label{E3.0.1}
0\to P_{d}\to \cdots P_{1}\to P_{0}\to k_B\to 0
\tag{E3.0.1}
\end{equation}
 where $P_{w}=\oplus_{s=1}^{n_w}B(-i_{w,s})$ for some
finite integers 
 $n_w$ and $i_{w,s}$. The Gorenstein condition (AS2)
implies that 
the above free resolution is symmetric in the 
 sense that the dual complex of \eqref{E3.0.1} is a
free resolution 
of the trivial right $B$-module (after a degree
shift). As a consequence,
we have $P_0=B$, $P_{d}=B(-l)$, $n_w=n_{d-w}$, and
$i_{w,s}+
i_{d-w, n_w-s+1}=l$ for all $w,s$.

For simplicity we only consider graded algebras generated in 
degree 1. Regular algebras of dimension 3 were classified by 
Artin, Schelter, Tate and Van den Bergh \cite{AS,ATV1,ATV2}. 
If $B$ is a regular algebra of dimension 3, then it is generated by 
either two or three elements. If $B$ is a noetherian regular 
algebra of global dimension 4, then $B$ is generated by either 
2, or 3 or 4 elements \cite[Proposition 1.4]{LPWZ}. Possible
forms of the minimal free resolutions of the trivial module $k$ 
are listed in \cite[Proposition 1.4]{LPWZ}. In this paper we 
only consider regular algebras $B$ generated by 4 elements of 
degree 1. Some discussions about the connection between the 
regularity of a right double extension and the invertibility
of $\sigma$ and $p_{12}$ are given in \cite[Section 2]{ZZ}.

Consider a sequence of graded algebra homomorphisms (sometimes 
called an ``exact sequence'' of algebras)
$$1\to A\to B\to C\to 1$$
so that 
\begin{enumerate}
\item[(i)]
$A$ is a subring of $B$, and 
\item[(ii)]
$C=B/BA_{\geq 1}$, where $BA_{\geq 1}=A_{\geq 1}B$ is
an ideal of $B$.
\end{enumerate}
This exact sequence will help us to prove
some homological properties of $B$ using
the properties of $A$ and $C$. By Proposition \ref{xxprop1.14},
a double extension gives rise to such an exact 
sequence (where $C= k\langle y_1,y_2\rangle/
(y_2y_1-p_{12}y_1y_2-p_{11}y_1^2)$).

For a trimmed double extension we can switch the position
of $A$ and $C$. To avoid the confusion we use different 
letters. Here is the setup. Let $D\subset B$ 
are connected graded rings satisfying the following conditions:
\begin{equation}
\label{E3.0.2}
\qquad
\tag{E3.0.2}
\end{equation}
\begin{enumerate}
\item
$B$ is free (or equivalently flat in this case) over
$D$ on the left and on the right. 
\item
$D_{\geq 1}B=BD_{\geq 1}$, hence it is a two-sided 
ideal of $B$. Let $C$ be the graded factor ring 
$B/D_{\geq 1}B$
\item
The trivial $D$-module $_Dk$ has a free resolution
$P^*$:
$$\cdots \to P^{-i}\to P^{-i+1}\to \cdots\to P^0\to
k\to 0$$ 
such that each $P^{-i}$ is a finitely generated free $D$-module.
\end{enumerate}
Condition (E3.0.2)(a) holds when $B$ is a double
extension of $D=A$. However, $D$ could be different
from $A$ in general. 

\begin{lemma}
\label{xxlem3.1} 
Assume \textup{(E3.0.2)}. Let $k$ denote 
the trivial module over $D$ (or over $B$, or over $C$).
\begin{enumerate}
\item
$B\otimes^L_D k\cong C$ as complex of graded
$B$-modules. 
\item
If $M$ is a bounded complex of graded $D$-module 
with finite projective dimension, then 
$\RHom_D(k,M)=\RHom_{D}(k,D)\otimes^L_D M$.
\item
$\RHom_D(k,B)=\RHom_{D}(k,D)\otimes_D B$.
\item
$\RHom_B(C,k)\cong \RHom_{D}(k,k)$ as a complex of
graded $k$-modules. As a consequence, $\pdim_B
C=\gldim D$.
\item
Suppose the algebra $D$ satisfies \textup{(AS2)}
with $\RHom_D(k,D)=k(l_D)[d_D]$ for 
some integers $l_D, d_D\in {\mathbb Z}$. 
Then $\RHom_B(C,B)\cong C(l_D)[d_D]$ as complexes of
right $B$-modules.
\end{enumerate}
\end{lemma}

\begin{proof} (a) Since $B_D$ is free, $B\otimes^L_D
k\cong
B\otimes_D k=B/BD_{\geq 1}=C$.

(b) We replace $M$ by a bounded complex of projective
$D$-modules,
still denoted by $M$.
Let $P^*\to k\to 0$ be a free resolution of $_Dk$ such
that
each $P^{-i}$ is finite. Then we have
$$\RHom_D(k,M)\cong \Hom_D(P^*,M)\cong
\Hom_D(P^*,D)\otimes_D M$$
where the last $\cong$ follows by the fact each
$P^{-i}$ is finite.
We continue the computation
$$\Hom_D(P^*,D)\otimes_D M\cong \RHom_D(k,D)\otimes_D
M\cong
\RHom_D(k,D)\otimes^L_D M.$$

(c) Note that $_DB$ is free, so the assertion is a
consequence of (b).

(d) We use part (a) and the $\Hom$-$\otimes$
adjunction
$$\RHom_B(C,{_Bk})\cong \RHom_B(B\otimes^L_D k,k)
\qquad\qquad\qquad\qquad\qquad$$
$$\qquad\qquad\qquad\qquad\cong
\RHom_{D}(k,\RHom_{B}(B,k))\cong \RHom_D(k,k).$$

(e) We use part (a) and the $\Hom$-$\otimes$
adjunction
$$\RHom_B(C,B)\cong \RHom_B(B\otimes^L_D k,B)
\qquad\qquad\qquad\qquad\qquad$$
$$\qquad\qquad\qquad\qquad \cong
\RHom_{D}(k,\RHom_{B}(B,B))\cong \RHom_D(k,B).$$
Next we use part (c) and the Gorenstein property (AS2)
$$\RHom_D(k,B)\cong \RHom_D(k,D)\otimes_D B\cong
k(l_D)[d_D]
\otimes_D B \cong C(l_D)[d_D].$$
\end{proof}

By Lemma \ref{xxlem3.1}(e), $\RHom_B(C,B)\cong C(l_D)[d_D]$
as complexes of right $B$-modules. But we don't know
at this point whether this is an isomorphism of 
complexes of left $B$-modules.

\begin{proposition}
\label{xxprop3.2} 
Assume \textup{(E3.0.2)}. In parts
(b,c) further assume that 
\begin{equation}
\label{E3.2.1}
\RHom_B(C,B)\cong C(l_D)[d_D]
\tag{E3.2.1}
\end{equation} 
as left $C$-module.
\begin{enumerate}
\item
$\gldim B\leq \gldim D+\gldim C$. 
\item
If $D$ and $C$ satisfy \textup{(AS2)}, then so 
does $B$. Further,
$l_D+l_C=l_B$ and $d_D+d_C=d_B$.
\item
If $D$ and $C$ satisfy \textup{(AS1,2)}, then so does $B$. 
In this case $\gldim B =\gldim D+\gldim C$.
\end{enumerate}
\end{proposition}

\begin{proof} (a) 
To prove the inequality we may assume that both
$\gldim D$ and $\gldim C$ are finite. By \eqref{E3.0.2}(c),
$\Ext^*_D(k,k)$ is finite dimensional. 

Let $X$ be a complex of graded $C$-module such that $H^*(X)$ 
is finite dimensional over $k$, then by induction on
the $k$-vector space dimension of $H^*(X)$, we have that
$\Ext^n_C(k,X)=0$ 
for $n>\gldim C+\sup X$. When $X=\RHom_D(k,k)$ (with the
$C$-module structure given by Lemma
\ref{xxlem3.1}(d)),
then $\Ext^n_C(k,\RHom_D(k,k))=0$ for all $n>\gldim
C+\gldim D$
because $\sup \RHom_D(k,k)=\gldim D$.

Viewing $C$ as a $(B,C)$-bimodule. We
have $_Bk\cong C\otimes^L_C {_Ck}$. By
$\Hom$-$\otimes$
adjunction we have
$$\RHom_B(k,k)\cong \RHom_B(C\otimes^L_C {_Ck},k)
\cong \RHom_C(k,\RHom_B(C,k)).$$
By Lemma \ref{xxlem3.1}(c), we have
$$\RHom_C(k,\RHom_B(C,k))\cong
\RHom_C(k,\RHom_D(k,k)).$$
Therefore
$$\Ext^n_B(k,k)\cong \Ext^n_C(k,\RHom_D(k,k))=0$$
for all $n>\gldim C+\gldim D$. This shows the
inequality.

(b) We use $\Hom$-$\otimes$ adjunction again
$$\RHom_B(k,B)\cong \RHom_B(C\otimes^L_C k, B)
\cong \RHom_{C}(k,\RHom_B(C,B)).$$
By hypothesis, $\RHom_B(C,B)\cong C(l_D)[d_D]$ as a
left $C$-module.
Hence 
$$\RHom_{C}(k,\RHom_B(C,B))\cong
\RHom_{C}(k,C)(l_D)[d_D]
\cong k(l_D+l_C)[d_D+d_C].$$
The assertion follows.

(c) This is a consequence of (a) and (b).
\end{proof}

\begin{theorem}
\label{xxthm3.3} Let $A$ be an Artin-Schelter regular algebra 
and let $B$ be a connected graded trimmed double extension 
$A_{P}[y_1,y_2;\sigma]$. Then $B$ is Artin-Schelter regular
and $\gldim B=\gldim A+2$.
\end{theorem}

\begin{proof} Let $C=A$, which is Artin-Schelter regular by
hypothesis. Let $D$ be the subalgebra $k_{P}[y_1,y_2]$. It's 
easy to see \eqref{E3.0.2}(b,c). Now we check \eqref{E3.0.2}(a). 
By symmetry, we only work on the right $D$-module $B$. Since $B$ is 
non-negatively graded, the projectivity of $B_D$ is equivalent to
the condition $\Tor_i^D(B,k)=0$ for all $i\neq 0$, which follows
from the exactness of \eqref{E2.2.2}.
 
By Theorem \ref{xxthm2.2}(b),  \eqref{E3.2.1} 
holds. By Proposition \ref{xxprop3.2} $B$ satisfies (AS1,2) and
$\gldim B=\gldim A+2$. By \eqref{E1.13.2}
$\GKdim B=\GKdim A+2$. Therefore $B$ is Artin-Schelter regular.
\end{proof}

To prove Theorem \ref{xxthm0.2} we need to pass the 
Artin-Schelter regularity from the trimmed double
extension $A_P[y_1,y_2;\sigma]$
to $A_P[y_1,y_2;\sigma,\delta,\tau]$. Note
that we don't know if $A$ and $B$ are noetherian 
in this setting. 

Let $A_P[y_1,y_2;\sigma,\delta,\tau]$ be a graded 
(or ungraded) double extension of $A$ with $d_1=\deg y_1$
and $d_2=\deg y_2$ (or $\deg y_1=\deg y_2=0$). Define a 
new grading $\newdeg y_1=d_1+1$ and $\newdeg y_2=d_2+1$ 
and $\newdeg x=\deg x$ for all $x\in A$. Using this 
grading we can define a filtration of $B$ by 
$$F_m B=\{\sum a_{n_1,n_2}y_1^{n_1}y_2^{n_2}\in B\;|\; 
\newdeg a_{n_1,n_2}+n_1\newdeg y_1+n_2\newdeg y_2\leq m\}.$$

\begin{lemma}
\label{xxlem3.4} Let $B$ be a graded (right) double 
extension $A_P[y_1,y_2;\sigma,\delta,\tau]$ of $A$.
\begin{enumerate}
\item
$\{F_m B\;|\; m\in {\mathbb Z}\}$ is a non-negative 
filtration such that the associated
graded ring $\gr_F B$ is isomorphic to $A_P[y_1,y_2;\sigma]$.
\item
Let $R_FB$ be the Rees ring associated to this filtration.
Then there is a central element $t$ of degree 1 such that
$R_FB/(t)=A_P[y_1,y_2;\sigma]$ as graded rings and 
$R_FB/(t-1)=A_P[y_1,y_2;\sigma,\delta,\tau]$ as ungraded 
rings.
\item
If $B$ is connected graded, then so are $\gr_FB$ and $R_FB$.
\end{enumerate}
\end{lemma}

\begin{proof} (a) Clearly $F_mB$ is a non-negative
filtration such that $B=\bigcup_m F_mB$. Since every
element in $B$ is of the form $\sum_{n_1,n_2}a_{n_1,n_2}y_1^{n_1}
y_2^{n_2}$, $\gr_F B$ is generated by $A$ and $y_1,y_2$.
The relation (R1) becomes
$$y_2y_1=p_{12}y_1y_2+p_{11}y_1^2$$
in $\gr_F B$ and (R2) becomes 
$$\begin{pmatrix}y_1\\y_2\end{pmatrix}r=\sigma(r)
\begin{pmatrix}y_1\\y_2\end{pmatrix}$$
in $\gr_FB$ which implies Definition \ref{xxdefn1.3}(aiv).
To prove $\gr_FB$ is a graded right double extension it suffices
to show Definition \ref{xxdefn1.3}(aiii). Since $\gr_FB$ is graded
we only need to show a graded version of Definition 
\ref{xxdefn1.3}(aiii). Say, $\sum_{n_1,n_2}a_{n_1,n_2}y_1^{n_1}
y_2^{n_2}=0$ in $\gr_F B$ for some homogeneous elements 
$a_{n_1,n_2}\in A$. Then 
$$\sum_{n_1,n_2}a_{n_1,n_2}y_1^{n_1}
y_2^{n_2}=\sum_{n_1,n_2}b_{n_1,n_2}y_1^{n_1}
y_2^{n_2}$$
with $\newdeg (\sum_{n_1,n_2}b_{n_1,n_2}y_1^{n_1}
y_2^{n_2})<\newdeg (\sum_{n_1,n_2}a_{n_1,n_2}y_1^{n_1}
y_2^{n_2})$. But this is impossible unless $a_{n_1,n_2}=b_{n_1,n_2}=0$
for all $n_1,n_2$. Thus we proved Definition \ref{xxdefn1.3}(aiii)
for $\gr_FB$ and hence $\gr_FB$ is a right double extension. 

If $B$ is a double extension, $\gr_FB$ is also a double extension
of $A$. The assertion follows.

(b,c) Clear.
\end{proof}

If $A_P[y_1,y_2;\sigma]$ is noetherian, then it
is well-known that the Artin-Schelter regularity passes
from $\gr_FB=A_P[y_1,y_2;\sigma]$ to 
$B:=A_P[y_1,y_2;\sigma,\delta,\tau]$. However we don't 
know if $A_P[y_1,y_2;\sigma]$ is always noetherian when 
$A$ is. So for the rest of section  we generalize some arguments
in the noetherian case to the non-noetherian case. 

\begin{lemma} 
\label{xxlem3.5}
Let $B$ be a connected graded ring with 
a central non-zero-divisor $t$ of positive degree. If $B/(t)$ is
Artin-Schelter regular, then so is $B$. Further
$\gldim B=\gldim B/(t)+1$.

As a consequence, if $\gr_FB:=A_P[y_1,y_2;\sigma]$ is 
Artin-Schelter regular, so is $R_FB$. 
\end{lemma}

\begin{proof} First of all $\GKdim B=\GKdim B/(t)+1<\infty$.
To prove (AS1,AS2) we set $D=k[t]$. Since $t$ is a central 
non-zero-divisor, \eqref{E3.0.2} holds. By Rees' lemma,
\eqref{E3.2.1} holds. The assertion follows from Proposition
\ref{xxprop3.2}.
\end{proof}

Next we want to show that if $\gr_F B$ (or $R_FB$) is 
Artin-Schelter regular and if $B$ is connected graded, 
then $B$ is Artin-Schelter regular.

\begin{theorem}
\label{xxthm3.6} Let $B$ be a connected graded algebra with a
filtration such that $\gr_FB$ is a connected graded Artin-Schelter
algebra. Then $B$ is Artin-Schelter regular and $\gldim B=
\gldim \gr_F B$.

As a consequence if $A_P[y_1,y_2;\sigma]$ is Artin-Schelter
regular, then so is a connected graded double extension 
$A_P[y_1,y_2;\sigma,\delta,\tau]$.
\end{theorem}

\begin{proof} Let $t$ be the central non-zero-divisor 
$1\in F_1B\subset R_FB$. Then we have $R_FB/(t)=\gr_FB$. Since 
$\gr_FB$ is Artin-Schelter regular, by Lemma \ref{xxlem3.5}, 
the Rees ring $R_FB$ is Artin-Schelter regular. Since
the localization $R_FB[t^{-1}]$ is isomorphic to 
$B[t^{\pm 1}]$, we have 
\begin{equation}
\label{E3.6.1}
\gldim B=\gr.\gldim B[t^{\pm 1}]\leq 
\gldim R_FB[t^{-1}]\leq \gldim R_FB<\infty.
\tag{E3.6.1}
\end{equation}
Thus (AS1) holds for $B$. If $t$ is a central non-zero-divisor in 
any algebra $A$, $\GKdim A=\GKdim A[t^{-1}]$ by \cite[8.2.13]{MR}. 
Hence, by \cite[8.2.7(iii)]{MR},
$$\GKdim B+1=\GKdim B[t^{\pm 1}]=\GKdim R_FB[t^{-1}]
=\GKdim R_FB<\infty.$$
Thus (AS3) holds for $B$. It remains to show (AS2) for $B$.

Let $M$ be a left $B$-module of $k$-dimension $d<\infty$. (For 
example, $M$ is the trivial $B$-module $B/B_{\geq 1}$). Let 
$N$ denote the graded vector space $\oplus_{n\geq 0}M$,
namely, $N_n=M$ for $n\geq 0$ and $N_n=0$ for $n<0$. We 
define a left graded $R_FB$-module structure as follows: for 
$x\in M=N_{n}$ and $a\in F_mB$, $a*x=ax\in M=(N)_{n+m}$.
Then we have an exact sequence of graded left $R_FB$-modules
\begin{equation}
\label{E3.6.2}
0\to N(-1)\xrightarrow{t} N\to k^{\oplus d}\to 0.
\tag{E3.6.2}
\end{equation}
Since $R_FB$ is Artin-Schelter regular, its trivial module
$k$ has a finite free resolution \eqref{E3.0.1}. This implies
that each $\Tor^{R_FB}_i(k,k)$ is finite dimensional for every 
$i$. We now consider the Tor group $\Tor^{R_FB}_i(k,N)$. Since $t$ 
is central, there is a $t$-action on $\Tor^{R_FB}_i(k,N)$ which is 
induced by the map in \eqref{E3.6.2}. Since $kt=0=tk$,
$t\Tor^{R_FB}_i(k,N)=0$. Applying $\Tor^{R_FB}(k,-)$ to the short sequence 
\eqref{E3.6.2}, we have a long exact sequence
$$\cdots\to
\Tor^{R_FB}_i(k,N(-1))\xrightarrow{0}\Tor^{R_FB}_i(k,N)
\to \Tor^{R_FB}_i(k,k^{\oplus d})\to \cdots$$
where the zero map is induced by the multiplication by $t$.
Hence $\Tor^{R_FB}_i(k,N)$ is a subspace of 
$\Tor^{R_FB}_i(k,k^{\oplus d})$ which is finite dimensional
over $k$. Hence the left $R_FB$-module $N$ has a finite 
graded free resolution. 
We compute $\Ext^i_{R_FB}(N,R_FB)$ by using a finite graded free 
resolution of $N$, then each $\Ext^i_{R_FB}(N,R_FB)$ is bounded 
below. Applying $\Ext_{R_FB}(-,R_FB)$ to \eqref{E3.6.2} we obtain 
a long exact sequence
$$\cdots\to \Ext^i_{R_FB}(k^{\oplus d},R_FB)\to \Ext^i_{R_FB}(N,R_FB)
\xrightarrow{f_i} \Ext^i_{R_FB}(N(-1),R_FB)\qquad$$
$$\to \Ext^{i+1}_{R_FB}(k^{\oplus d},R_FB)\to\cdots.
\qquad\qquad\qquad \qquad\qquad\qquad\qquad\qquad\qquad$$
where $f_i$ is induced by the multiplication by $t$. 
If $i\neq \gldim R_FB-1$, then $\Ext^{i+1}_{R_FB}(k^{\oplus d},R_FB)=0$.
The map $f_i=0$ since $\Ext^i_{R_FB}(N,R_FB)$ is bounded below.
When $i=\gldim R_FB-1:=p$, then we have 
$$0\to \Ext^p_{R_FB}(N,R_FB)
\xrightarrow{f_p} \Ext^p_{R_FB}(N(-1),R_FB)\to
k^{\oplus d}(l)\to 0.$$
Hence $f_p$ is injective and $\Ext^p_{R_FB}(N,R_FB)$ has
the Hilbert series equal to $dt^{-l}(1-t)^{-1}$. This implies
that the localization $\Ext^p_{R_FB}(N,R_FB)[t^{-1}]=X[t^{\pm 1}]$
for some $d$-dimensional left $B$-module $X$.

Since $N$ has a finite free resolution (hence pseudo-coherent 
in the sense of \cite{YZ}), by \cite[Lemma 2.2(a)]{YZ},
we have, for every $j$,
$$\Ext^j_{R_FB}(N,R_FB)[t^{-1}]
\cong \Ext^j_{R_FB}(N,R_FB[t^{-1}]),$$
and by \cite[Lemma 2.2(b)]{YZ}
$$\Ext^j_{R_FB}(N,R_FB[t^{-1}])\cong 
\Ext^j_{R_FB[t^{-1}]}(N[t^{-1}],R_FB[t^{-1}]).$$
Since $N[t^{-1}]=M[t^{\pm t}]$ and $\grmod 
R_FB[t^{-1}]\cong \mod B$, we have
$$\Ext^j_{R_FB[t^{-1}]}(N[t^{-1}],R_FB[t^{-1}])
\cong \Ext^j_{B}(M, B)[t^{\pm 1}].$$
Combining the results we proved in the last paragraph, we have
$$\Ext^j_{B}(M, B)[t^{\pm 1}]=\begin{cases} 0& i\neq \gldim R_FB-1 \\
X[t^{\pm 1}]&i=\gldim R_FB-1 \end{cases}.$$
Hence 
$$\Ext^j_{B}(M, B)=\begin{cases} 0& i\neq \gldim R_FB-1 \\
X&i=\gldim R_FB-1 \end{cases}.$$
If we take $M$ to be the trivial module $k=B/B_{\geq 1}$, 
this is (AS2) for $B$ and $\gldim B\geq \gldim R_FB-1$. 

To complete (AS2) we need to check that $\gldim B=\gldim R_FB-1$.
By  \eqref{E3.6.1}, $\gldim B\leq \gldim R_FB=r$. Consider the minimal
graded free resolution of the trivial $B$-module $k$,
$$0\to P^{-r}\to P^{-r+1}\to \cdots\to P^{-1}\to B\to k\to 0.$$
The proof of \cite[Proposition 3.1(a)]{SZ} shows that
if $\Ext^r_B(k,B)=0$ then $P^{-r}=0$. Thus $\gldim B=r-1$
as desired.

The consequence follows from the fact that
$\gr A_{P}[y_1,y_2;\sigma,\delta,\tau]=A_{P}[y_1,y_2;\sigma]$.
\end{proof}

Now we are ready to prove Theorem \ref{xxthm0.2}.

\begin{proof}[Proof of Theorem \ref{xxthm0.2}]
Let $A$ be an Artin-Schelter regular algebra
and let $B=A_{P}[y_1,y_2;\sigma,\delta,\tau]$ be a 
connected graded double extension of $A$. Let $C$ be
the trimmed double extension $A_{P}[y_1,y_2;\sigma]$. 
By Theorem \ref{xxthm3.3}, $C$ is Artin-Schelter regular
and $\gldim C=\gldim A+2$. By Theorem \ref{xxthm3.6}, 
$B$ is Artin-Schelter regular and $\gldim B=\gldim C$.
\end{proof}

\section{Examples and questions}
\label{xxsec4}

The first example is fairly easy. 

\begin{example}
\label{xxex4.1} Let $A=k[x]$. There are so many different right 
double extensions of $A$. It is possible to classify 
all connected graded double extensions $A_{P}[y_1,y_2;\sigma,
\delta,\tau]$ with $\deg x=\deg y_1=\deg y_2=1$, but it 
makes no sense to list all since these algebras are 
all well-known regular algebra of global dimension 3 studied by 
Artin and Schelter \cite{AS}. We only give a few examples to 
indicate that the DE-data can be different kinds.

The relations (R1) and (R2) can be written as 
$$\begin{aligned}
y_2y_1&=p_{12}y_1y_2+p_{11}y_1^2+a xy_1+b xy_2+c x^2,\\
y_1 x &=d xy_1+e xy_2+ f x^2,\\
y_2 x &=g xy_1+h xy_2+ i x^2.
\end{aligned}
$$
The relations (R3) are equivalent to the fact that the overlap 
between the above three relations can be resolved. Since 
$p_{12}\neq 0$, we may choose $P$ to be $(p,0)$ or $(1,1)$. 
The invertibility of $\sigma$ is equivalent to the condition 
$dh-ge\neq 0$. The details of (R3) are omitted here. Here
are a few examples which we only list the defining relations
of $B$. Recall that $B$ is generated by $x,y_1,y_2$.

\noindent
{\bf Algebra $B^1:=B^1(p,a,b,c)$}, where $a,b,c,p\in k$ and
$p\neq 0, b\neq 0,1$. 
$$\begin{aligned}
y_2y_1&=py_1y_2+{\frac{bc}{1-b}}(pb-1)xy_1+a x^2,\\
y_1 x &=b xy_1,\\
y_2 x &=b^{-1} xy_2+ c x^2.
\end{aligned}
$$
The homomorphism $\sigma$ is determined by $\sigma(x)=
\begin{pmatrix} bx&0\\0&b^{-1}x\end{pmatrix}$. In this case
$\sigma_{12}=\sigma_{21}=0$ and $\sigma_{11}$ and
$\sigma_{22}$ are algebra automorphisms. The derivation
is determined by $\delta(x)=\begin{pmatrix} 0\\cx^2\end{pmatrix}$.
The parameter $P$ is $(p,0)$ and the tail $\tau$ is $\{
{\frac{bc}{1-b}}(pb-1)x,0,ax^2\}$. We can use the linear 
transformation $y_2\to y_2+\frac{bc}{b-1}x$ to make $c=0$.
This means that $B^1(1,p,a,b,c)\cong B^1(1,p,a,b,0)$.

\noindent
{\bf Algebra $B^2:=B^2(a,b,c)$}, where $a,b,c\in k$ and $b\neq 0$.
$$\begin{aligned}
y_2y_1&=-y_1y_2+a x^2,\\
y_1 x &=b^{-1} xy_2+ c x^2,\\
y_2 x &=bxy_1+ (-bc) x^2.
\end{aligned}
$$
The homomorphism $\sigma$ is determined by $\sigma(x)=
\begin{pmatrix} 0&b^{-1}x\\bx&0\end{pmatrix}$. By induction,
one sees that $\sigma(x^n)=
\begin{pmatrix} 0&b^{-1}x^n\\bx^n&0\end{pmatrix}$ when $n$ is odd
and $\sigma(x^n)=
\begin{pmatrix} x^n&0\\0&x^n\end{pmatrix}$ when $n$ is even.
In this case $\sigma_{11}$
and $\sigma_{22}$ are not algebra homomorphisms of $A$.  The derivation
is determined by $\delta(x)=\begin{pmatrix} cx^2\\-bcx^2\end{pmatrix}$.
The parameter $P$ is $(-1,0)$ and the tail $\tau$ is $\{0,0,ax^2\}$.
It is easy to see that $B^2(a,b,c)$ is isomorphic to
$B^2(ab^{-1},1,c)$.

For most $(a,b,c)$ (for example for $(a,b,c)=(1,1,0)$), the algebra 
$B^2(a,b,c)$ is not an Ore extension of any regular algebra of 
global dimension 2. One can prove this assertion by direct computations 
and the fact that $x$ is normalizing. Therefore $B^2(a,b,c)$ is not
an iterated Ore extension of $k[x]$.

\noindent
{\bf Algebra $B^3:=B^3(a)$}, where $a\in k$ and $a\neq 0$.
$$\begin{aligned}
y_2y_1&=y_1y_2,\\
y_1 x &=axy_1+ xy_2,\\
y_2 x &=axy_2.
\end{aligned}
$$
The homomorphism $\sigma$ is determined by $\sigma(x)=
\begin{pmatrix} ax&x\\0&ax\end{pmatrix}$. This $\sigma$ is
``upper-triangular''. The derivation
is zero. The parameter $P$ is $(1,0)$ and the tail 
$\tau$ is $\{0,0,0\}$.

\noindent
{\bf Algebra $B^4:=B^4(a,b,c)$}, where $a,b,c\in k$ and $b\neq -1$.
$$\begin{aligned}
y_2y_1&=y_1y_2+y_1^2+axy_1+\frac{b}{1+b}xy_2+cx^2,\\
y_1 x &=xy_1+ bx^2,\\
y_2 x &=(2+2b^{-1})xy_1+xy_2.
\end{aligned}
$$
The homomorphism $\sigma$ is determined by $\sigma(x)=
\begin{pmatrix} x&0\\(2+2b^{-1})x&x\end{pmatrix}$. 
The derivation is determined by 
$\delta(x)=\begin{pmatrix} bx^2\\0\end{pmatrix}$. 
The parameter $P$ is $(1,1)$ and the tail 
$\tau$ is $\{ax,\frac{b}{1+b}x,cx^2\}$.

One can see from these examples that the DE-data can be
all different even starting with a very simple 
algebra $A=k[x]$.

Let $B$ denote any of the algebras $B^i$. The element $x$ is 
normalizing in $B$ and $B/(x)$ is a regular algebra of global dimension 2. 
These algebras are called normal extensions. Except for 
$B^2$, these algebras are also iterated Ore extensions of
$k[x]$. There are possible isomorphisms between these algebras
and we refer to \cite{AS, ATV1, ATV2} for the classification and 
general properties of Artin-Schelter regular algebras of global 
dimension 3.
\end{example}

The second example seems new. It is Artin-Schelter 
regular of global dimension 4. Many other algebras are 
given in \cite{ZZ}.

\begin{example}
\label{xxex4.2}
Let $h$ be any nonzero scalar in $k$. Let $B(h)$ denote 
the graded algebra generated by $x_1,x_2,y_1,y_2$ and
subject to the following conditions
\begin{equation}
\label{E4.2.1}
\begin{aligned}
x_2x_1&=-x_1x_2\\
y_2y_1&=-y_1y_2\\
y_1x_1&=h(x_1y_1+     x_2y_1+     x_1y_2\quad \qquad)\\
y_1x_2&=h(\qquad\qquad\quad \;\;  +x_1y_2 \quad\qquad)\\
y_2x_1&=h(\qquad\; \; + x_2y_1 \quad \quad\qquad\qquad)\\
y_2x_2&=h(\qquad\;\;  -x_2y_1-x_1y_2+x_2y_2).
\end{aligned}
\tag{E4.2.1}
\end{equation}
If we use the notation in \cite{ZZ}, the set of the last four 
relations is associated to the matrix
$$\Sigma=h \begin{pmatrix} 1&1&1&0\\0&0&1&0\\
                           0&1&0&0\\0&-1&-1&1
\end{pmatrix}.$$
\end{example}

We will work out some properties of $B(h)$. Clearly
$B(h)$ is a connected graded algebra where $\deg x_1=
\deg x_2=\deg y_1=\deg y_2=1$. And it is quadratic.
We will see that it is Koszul. If we define
$\deg x_1=\deg x_2=(1,0)$ and $\deg y_1=\deg y_2=(0,1)$,
then $B(h)$ is ${\mathbb Z}^2$-graded. The following
lemma shows that $B(h)$ is very symmetric between $x_i$'s
and $y_i$'s.

\begin{lemma}
\label{xxlem4.3} Let $B(h)$ be defined as above.
\begin{enumerate}
\item
$B(h)$ has an automorphism determined by $f(x_1)=x_2$, 
$f(x_2)=-x_1$, $f(y_1)=y_2$ and $f(y_2)=-y_1$.
\item
$B(h)$ has an anti-automorphism determined by
$g(x_1)=y_1$, $g(x_2)=y_2$, $g(y_1)=x_1$ and $g(y_2)=x_2$.
\end{enumerate}
\end{lemma}

\begin{proof} Easy.
\end{proof}

Next we show that $B(h)$ is a graded double extension
of $A:=k_{-1}[x_1,x_2]=k\langle x_1,x_2\rangle/(x_1x_2+x_2x_1)$. 
First of all, $P=(-1,0)$ and $\delta=0$
and $\tau=(0,0,0)$. We only need to define algebra homomorphism 
$\sigma:A\to M_2(A)$. Since $A$ is generated by $x_1$ and $x_2$,
we only need to assign $\sigma(x_1)$ and $\sigma(x_2)$.

\begin{lemma}
\label{xxlem4.4} Let 
$\sigma(x_1)=h\begin{pmatrix} x_1+x_2&x_1\\x_2&0\end{pmatrix}$
and 
$\sigma(x_2)=h\begin{pmatrix} 0&x_1\\-x_2&-x_1+x_2\end{pmatrix}$.
Then $\sigma$ defines an algebra homomorphism $A\to M_2(A)$.
And if we write $\sigma=\begin{pmatrix} \sigma_{11}&\sigma_{12}
\\ \sigma_{21}&\sigma_{22}\end{pmatrix}$, then $\sigma_{ij}$
is not an algebra homomorphism for each pair $(i,j)$.
\end{lemma}

\begin{proof} For the first assertion we need to show that
$\sigma(x_1)\sigma(x_2)+\sigma(x_2)\sigma(x_1)=0$.
$$\begin{aligned}
h^{-2}(\sigma&(x_1)\sigma(x_2)+\sigma(x_2)\sigma(x_1))\\
&=
\begin{pmatrix} x_1+x_2&x_1\\x_2&0\end{pmatrix}
\begin{pmatrix} 0&x_1\\-x_2&-x_1+x_2\end{pmatrix}
+
\begin{pmatrix} 0&x_1\\-x_2&-x_1+x_2\end{pmatrix}
\begin{pmatrix} x_1+x_2&x_1\\x_2&0\end{pmatrix}\\
&=\begin{pmatrix} -x_1x_2&0\\0&x_2x_1\end{pmatrix}
+\begin{pmatrix} x_1x_2&0\\0&-x_2x_1\end{pmatrix}\\
&=0.
\end{aligned}
$$
So $\sigma$ is an algebra homomorphism $A\to M_2(A)$.
For the second assertion we note that 
$$\begin{pmatrix} \sigma_{11}(x_1)&\sigma_{12}(x_1)
\\ \sigma_{21}(x_1)&\sigma_{22}(x_1)\end{pmatrix}
=h\begin{pmatrix} x_1+x_2&x_1\\x_2&0\end{pmatrix}
$$
and
$$\begin{pmatrix} \sigma_{11}(x_1^2)&\sigma_{12}(x_1^2)
\\ \sigma_{21}(x_1^2)&\sigma_{22}(x_1^2)\end{pmatrix}
=h^2\begin{pmatrix} (x_1+x_2)^2+x_1x_2&(x_1+x_2)x_1\\x_2(x_1+x_2)&
x_2x_1\end{pmatrix}.
$$
Therefore $\sigma_{ij}(x_1^2)\neq \sigma_{ij}(x_1)^2$
for each pair $(i,j)$.
\end{proof}

Restricted to the generators we
$$\sigma_{11}\begin{pmatrix}x_1\\ x_2 \end{pmatrix}:=
\begin{pmatrix}\sigma_{11}(x_1)\\ \sigma_{11}(x_2)
\end{pmatrix}=h\begin{pmatrix} 1&1 \\0&0\end{pmatrix}
\begin{pmatrix}x_1\\ x_2 \end{pmatrix}
$$
$$\sigma_{12}\begin{pmatrix}x_1\\ x_2 \end{pmatrix}:=
\begin{pmatrix}\sigma_{12}(x_1)\\ \sigma_{12}(x_2)
\end{pmatrix}=h\begin{pmatrix} 1&0 \\1&0\end{pmatrix}
\begin{pmatrix}x_1\\ x_2 \end{pmatrix}
$$
$$\sigma_{21}\begin{pmatrix}x_1\\ x_2 \end{pmatrix}:=
\begin{pmatrix}\sigma_{21}(x_1)\\ \sigma_{21}(x_2)
\end{pmatrix}=h\begin{pmatrix} 0&1 \\0&-1\end{pmatrix}
\begin{pmatrix}x_1\\ x_2 \end{pmatrix}
$$
$$\sigma_{22}\begin{pmatrix}x_1\\ x_2 \end{pmatrix}:=
\begin{pmatrix}\sigma_{22}(x_1)\\ \sigma_{22}(x_2)
\end{pmatrix}=h\begin{pmatrix} 0&0 \\-1&1\end{pmatrix}
\begin{pmatrix}x_1\\ x_2 \end{pmatrix}
$$

\begin{lemma}
\label{xxlem4.5} The date $\{P,\sigma,\delta,\tau\}$
satisfies \textup{(R3)} for the generating set $\{x_1,x_2\}$.
As a consequence, the algebra $B(h)$ is a right
double extension of $A$.
\end{lemma}

\begin{proof} Since $\delta=0$ and $\tau=(0,0,0)$,
only the first three relations in (R3) are non-trivial, namely,
$$\begin{aligned}
\sigma_{21}\circ \sigma_{11}&=-\sigma_{11}\circ \sigma_{21}\\
\sigma_{21}\circ \sigma_{12}-\sigma_{22}\circ \sigma_{11}
&=-\sigma_{11}\circ\sigma_{22}+\sigma_{12}\circ\sigma_{21}\\
\sigma_{22}\circ \sigma_{12}&=-\sigma_{12}\circ \sigma_{22}
\end{aligned}
$$
We claim that these three relations are satisfied by the 
generating set $\{x_1,x_2\}$. For the first one, we have 
$$\sigma_{21}\circ \sigma_{11}\begin{pmatrix}x_1\\ x_2 \end{pmatrix}
=h^2\begin{pmatrix} 1&1 \\0&0\end{pmatrix}
\begin{pmatrix} 0&1 \\0&-1\end{pmatrix}
\begin{pmatrix}x_1\\ x_2 \end{pmatrix}=
\begin{pmatrix}0\\ 0\end{pmatrix}$$
and
$$-\sigma_{11}\circ \sigma_{21}
\begin{pmatrix}x_1\\ x_2 \end{pmatrix}
=-h^2\begin{pmatrix} 0&1 \\0&-1\end{pmatrix}
\begin{pmatrix} 1&1 \\0&1\end{pmatrix}
\begin{pmatrix}x_1\\ x_2 \end{pmatrix}=
\begin{pmatrix}0\\ 0\end{pmatrix}.$$
Hence
$$\sigma_{21}\circ \sigma_{11}
\begin{pmatrix}x_1\\ x_2 \end{pmatrix}=-\sigma_{11}\circ \sigma_{21}
\begin{pmatrix}x_1\\ x_2 \end{pmatrix}.$$
Similarly, for the third one, one has
$$\sigma_{22}\circ \sigma_{12}\begin{pmatrix}x_1\\ x_2 \end{pmatrix}
=0=-\sigma_{12}\circ \sigma_{22}
\begin{pmatrix}x_1\\ x_2 \end{pmatrix}.$$
For the second relation, a computation shows that
$$(\sigma_{21}\circ \sigma_{12}-\sigma_{22}\circ \sigma_{11})
\begin{pmatrix}x_1\\ x_2 \end{pmatrix}=
h^2\begin{pmatrix}x_1\\ x_2 \end{pmatrix}
$$
and
$$(-\sigma_{11}\circ\sigma_{22}+\sigma_{12}\circ\sigma_{21})
\begin{pmatrix}x_1\\ x_2 \end{pmatrix}=
h^2\begin{pmatrix}x_1\\ x_2 \end{pmatrix}.$$
Therefore
$$(\sigma_{21}\circ \sigma_{12}-\sigma_{22}\circ \sigma_{11})
\begin{pmatrix}x_1\\ x_2 \end{pmatrix}=
(-\sigma_{11}\circ\sigma_{22}+\sigma_{12}\circ\sigma_{21})
\begin{pmatrix}x_1\\ x_2 \end{pmatrix}.$$
Up to this point we have verified (R3) for the generators. 
The last four relations given in \eqref{E4.2.1} are (R2) for 
the generator. By Lemma \ref{xxlem4.4},
$\sigma$ is an algebra homomorphism and $\delta=0$ is 
trivially a $\sigma$-derivation. Therefore by
Proposition \ref{xxprop1.11}, $B$ is a right double extension. 
\end{proof}

\begin{proposition} 
\label{xxprop4.6}
\begin{enumerate}
\item
The determinant of $\sigma$ is equal to  
$$\det \sigma=\sigma_{11}\circ \sigma_{22}+\sigma_{21}\circ
\sigma_{12}
=\sigma_{22}\circ \sigma_{11}+\sigma_{12}\circ
\sigma_{21}$$ 
which is determined by $\det\sigma(x_1)=h^2 x_2$
and $\det\sigma(x_2)=-h^2 x_1$.
\item
$B(h)$ is a connected graded double extension of
$k_{-1}[x_1,x_2]$.
\item
$\det\sigma\neq \sigma_{22}\circ \sigma_{11}+\sigma_{21}\circ
\sigma_{12}$ and $\det \sigma\neq 
\sigma_{11}\circ \sigma_{22}+\sigma_{12}\circ
\sigma_{21}$.
\end{enumerate}
\end{proposition}

\begin{proof} (a) Since $P=(-1,0)$, the determinant
of $\sigma$ is as defined. 
$$
\begin{aligned}
\det\sigma \begin{pmatrix}x_1\\ x_2 \end{pmatrix}
&=(\sigma_{11}\circ \sigma_{22}+\sigma_{21}\circ
\sigma_{12})\begin{pmatrix}x_1\\ x_2 \end{pmatrix}\\
&=h^2(\begin{pmatrix}0&0\\ -1&1 \end{pmatrix}
\begin{pmatrix}1&1\\ 0&0 \end{pmatrix}+
\begin{pmatrix}1&0\\ 1&0 \end{pmatrix}
\begin{pmatrix}0&1\\ 0&-1 \end{pmatrix})
\begin{pmatrix}x_1\\ x_2 \end{pmatrix}\\
&=\begin{pmatrix}0&h^2\\ -h^2&0 \end{pmatrix}
\begin{pmatrix}x_1\\ x_2 \end{pmatrix}.
\end{aligned}
$$
Hence $\det \sigma$ is invertible. By Proposition 
\ref{xxprop2.1}(b), $\sigma$ is invertible. 

(b) By part (a) $\sigma$ is invertible. Since 
$p_{12}=-1\neq 0$ and since $B$ is a connected graded 
right double extension, by Proposition \ref{xxprop1.13}, 
$B$ is a double extension.

(c) This follows by the following computation.
$$
\begin{aligned}
(\sigma_{22}\circ \sigma_{11}+\sigma_{21}\circ
\sigma_{12})\begin{pmatrix}x_1\\ x_2 \end{pmatrix}
&=(\begin{pmatrix}h&h\\ 0&0 \end{pmatrix}
\begin{pmatrix}0&0\\ -h&h \end{pmatrix}+
\begin{pmatrix}h&0\\ h&0 \end{pmatrix}
\begin{pmatrix}0&h\\ 0&-h \end{pmatrix})
\begin{pmatrix}x_1\\ x_2 \end{pmatrix}\\
&=\begin{pmatrix}-h^2&2h^2\\ 0&h^2 \end{pmatrix}
\begin{pmatrix}x_1\\ x_2 \end{pmatrix}
\end{aligned}
$$
and
$$
\begin{aligned}
(\sigma_{11}\circ \sigma_{22}+\sigma_{12}\circ
\sigma_{21})\begin{pmatrix}x_1\\ x_2 \end{pmatrix}
&=(\begin{pmatrix}0&0\\ -h&h \end{pmatrix}
\begin{pmatrix}h&h\\ 0&0 \end{pmatrix}+
\begin{pmatrix}0&h\\ 0&-h \end{pmatrix}
\begin{pmatrix}h&0\\ h&0 \end{pmatrix})
\begin{pmatrix}x_1\\ x_2 \end{pmatrix}\\
&=\begin{pmatrix}h^2&0\\ -2h^2&-h^2 \end{pmatrix}
\begin{pmatrix}x_1\\ x_2 \end{pmatrix}.
\end{aligned}
$$
So neither of $\sigma_{22}\circ \sigma_{11}+\sigma_{21}\circ
\sigma_{12}$ nor $\sigma_{11}\circ \sigma_{22}+\sigma_{12}\circ
\sigma_{21}$ is an algebra homomorphism of $A$, neither is
equal to $\det \sigma$.
\end{proof}

\begin{corollary}
\label{xxcor4.7} 
Let $B=B(h)$ be defined as in Example \ref{xxex4.2}.
\begin{enumerate}
\item
$B$ is strongly noetherian and has enough normal elements.
\item
$B$ is Koszul, Artin-Schelter regular of global dimension
4, generated by 4 elements of degree 1.
\item
$B$ is an Auslander regular and Cohen-Macaulay domain.
\item
The quotient division ring of $B$ is generated by 
two elements.
\end{enumerate}
\end{corollary}

\begin{proof} 
(a) We claim that $x_1x_2$ and $x_1^2+x_2^2$ are normal
elements of $B$. Since $x_1$ and $x_2$ are skew-commutative,
$x_1x_2$ is skew-commuting with $x_1$ and $x_2$. Similarly,
$x_1^2+x_2^2$ are commuting with $x_1$ and $x_2$. So we only 
need to check that these are (skew-)commuting with $y_1,y_2$. 
Use the relations in \eqref{E4.2.1}, 
we have
$$y_1(x_1x_2)=y_1(-x_2x_1)=-hx_1y_2x_1=-h^2x_1x_2y_1=(x_1x_2)(-h^2y_1)$$
and
$$y_2(x_1x_2)=hx_2y_1x_2=h^2x_2x_1y_2=-h^2x_1x_2y_2=(x_1x_2)(-h^2y_2).$$
Therefore $x_1x_2$ is normal. By using the third and 
fourth relations in \eqref{E4.2.1}, we have
\begin{equation}
\label{E4.7.1}
y_1(x_1-x_2)=h(x_1+x_2)y_1
\tag{E4.7.1}
\end{equation}
which implies that
$$y_1(x_1-x_2)^2=(x_1+x_2)^2(h^2y_1).$$
Similarly, we have
$$y_2(x_1+x_2)=h(-x_1+x_2)y_2,\quad
{\text{and}}\quad
y_2(x_1+x_2)^2=(x_1-x_2)^2(h^2y_2).$$
Since $(x_1-x_2)^2=(x_1+x_2)^2=x_1^2+x_2^2$, it is a normal
element of $B$.
By symmetry, $y_1y_2$ and $y_1^2+y_2^2$ are normal elements too. 
The factor ring
$B/(x_1x_2,x_1^2+x_2^2, y_1y_2,y_1^2+y_2^2)$
is a finite dimensional algebra. Every finite
dimensional algebra is strongly (and universally) noetherian. 
By \cite[Section 4]{ASZ}, $B$ is strongly (and universally) 
noetherian. Also $B$ has enough normal elements in the 
sense of \cite[p.392]{Zh}. 

(b) First of all $B$ is generated by $x_1,x_2,y_1,
y_2$ of degree 1. By Proposition \ref{xxprop4.6}(b), $B$ is
a connected graded double extension of an Artin-Schelter
regular algebra $A$ of global dimension 2. By Theorem 
\ref{xxthm3.3}, $B$ is Artin-Schelter regular of 4. 
By \eqref{E1.13.2} the Hilbert series of $B$ is
equal to $(1-t)^{-4}$. It follows from \cite[Theorem 5.11]{Sm}
that $B$ is Koszul.

(c) This follows from part (b) and \cite[Theorem 1]{Zh}.

(d) We claim that the quotient division ring $Q$ of $B$
is generated by $y_1$ and $x_1-x_2$. Let $D$ be the division
subalgebra of $Q$ generated by $y_1$ and $x_1-x_2$. It follows
from \eqref{E4.7.1} that $x_1+x_2\in D$. Hence both $x_1$
and $x_2$ are in $D$. The fourth quadratic relation of
$B$ (see Example \ref{xxex4.2}) implies that $y_2\in D$.
Therefore $D=Q$.
\end{proof}

\begin{corollary}
\label{xxcor4.8} 
Let $B=B(h)$. The following are equivalent:
\begin{enumerate}
\item[(i)]
$B$ is PI.
\item[(ii)]
$h$ is a root of unity.
\item[(iii)]
$\det\sigma$ is of finite order.
\end{enumerate}
\end{corollary}

\begin{proof}
(ii) $\Leftrightarrow$ (iii) Follows from Proposition
\ref{xxprop4.6}(a).

(i) $\Leftrightarrow$ (ii)
Suppose $h$ is a root of unity, say $l$th root of unity.
Then by the computation in the proof of Corollary \ref{xxcor4.7}(a), 
$(x_1x_2)^{2l}$, $(x_1^2+x_2^2)^l$, $(y_1y_2)^{2l}$ and 
$(y_1^2+y_2)^l$ are central elements of $B$. The factor
ring $B/((x_1x_2)^{2l}, (x_1^2+x_2^2)^l, (y_1y_2)^{2l},
(y_1^2+y_2)^l)$ is finite dimensional over $k$. Hence $B$ is PI.

If $h$ is not a root of unity, then $B$ contains a subalgebra
generated by $y_1$ and $(x_1x_2)$ with a relation $y_1(x_1x_2)=-h^2
(x_1x_2)y_1$, which is not PI. So $B$ is not PI.
\end{proof}

There are many examples of regular algebras of dimension
4. If $A$ is a regular algebra of dimension 3, then any Ore 
extension $A[y;\sigma,\delta]$ (with invertible $\sigma$) 
is a regular algebra of dimension 4.  Ore extensions such 
as $A[y;\sigma,\delta]$ are well-studied. A connected 
graded algebra $A$ is called a {\it normal extension} of a 
regular algebra fo dimension 3, if there is a normalizing 
non-zero-divisor $x$ of degree  1 such that $A/(x)$ is 
regular. Normal extension of regular algebras of dimension 
3 are regular of dimension 4 and studied by Le Bruyn, Smith 
and Van den Bergh \cite{LSV}. So we assume that normal 
extensions are well-understood. Next we show that $B(h)$ 
is neither an Ore extension nor a normal extension.

\begin{lemma}
\label{xxlem4.9} The algebra $B(h)$ has no normal element 
in degree 1. As a consequence, $B(h)$ is not a normal extension
of an Artin-Schelter regular algebra of global 
dimension 3.
\end{lemma}

\begin{proof} Let $z=ax_1+bx_2+cy_1+dy_2$ be a nonzero
normal element of $B:=B(h)$. Applying Lemma \ref{xxlem4.3} we 
may assume that $a\neq 0$, so after a scalar change we
may assume that $z=x_1+bx_2+cy_1+dy_2$. Since $z$ is not
in the ideal generated by $y_1$ and $y_2$, $z$ remains nonzero
and normal in $B/(y_1,y_2)$. Since $B/(y_1,y_2)\cong
k_{-1}[x_1,x_2]$, it is easy to see that $x_1+bx_2$ is normal
if and only if $b=0$. Thus $z=x_1+cy_1+dy_2$. Also it
is easy to check that $x_1$ in not normal in $B$. This
implies that $cy_1+dy_2\neq 0$. Since $x_1+cy_1+dy_2$
is normal,
$$x_1(x_1+cy_1+dy_2)=(x_1+cy_1+dy_2)(ex_1+fx_2+gy_1+hy_2).$$
Since $x_1x_1=x_1x_1$ in $B/(y_1,y_2)$, $e=1$ and $f=0$.
Modulo $x_1,x_2$, we have 
$$0=(cy_1+dy_2)(gy_1+hy_2)$$
which implies that $gy_1+hy_2=0$ because $cy_1+dy_2$
is not zero in the domain $B/(x_1,x_2)$. Thus
$$
\begin{aligned}
x_1(x_1+cy_1+dy_2)&=(x_1+cy_1+dy_2)x_1\\
&=x_1^2+ch(x_1y_1+x_2y_1+x_1y_2)+dhx_2y_1
\end{aligned}
$$
which implies that
$$c=ch,ch+dh=0,d=ch.$$
So $c=d=0$, a contradiction. Therefore $B$ has no
normal element in degree 1.

The consequence is clear.
\end{proof}

\begin{lemma}
\label{xxlem4.10} 
The algebra $B(h)$ does not contain a graded subalgebra
that has Hilbert series $(1-t)^{-3}$. As a consequence,
$B(h)$ is not an Ore extension of an Artin-Schelter
regular algebra generated by three elements in degree 1.
\end{lemma}

\begin{proof} Let $A$ be the subalgebra generated by
$z_1,z_2,z_3\in B_1$. By Lemma \ref{xxlem4.3}, we may 
assume $y_2$ is not in the subspace $kz_1+kz_2+kz_3$. 
This means that, up to a linear transformation, 
$z_1=x_1+ay_2,z_2=x_2+by_2,z_3=y_1+cy_2$. 
We order $x_1<x_2<y_1<y_2$ and use lexgraphic order for
monomials of higher degrees.  Let us compute 
some elements in the degree two part of $A$ as follows.
$$
\begin{aligned}
z_1^2&=x_1^2+ax_1y_2+ay_2x_1+a^2y_2^2=x_1^2+ax_1y_2+ax_2y_1+a^2y_2^2\\
z_1z_2&=x_1x_2+bx_1y_2+ay_2x_2+aby_2^2=x_1x_2+{\text{higher terms}} \\
z_2^2&=x_2^2+{\text{higher terms}}\\
z_1z_3&=x_1y_1+{\text{higher terms}}\\
z_2z_3&=x_2y_1+{\text{higher terms}}\\
z_3z_1&=(y_1+cy_2)x_1+{\text{higher terms}}
=hx_1y_1+hx_1y_2+{\text{higher terms}}\\
z_3^2&=(y_1+cy_2)^2=y_1^2+{\text{higher terms}}
\end{aligned}
$$
Hence these 7 elements are linearly independent and
hence $H_A(t)\neq (1-t)^{-3}$.

The consequence is clear. 
\end{proof}

Combining the above statements we have Proposition \ref{xxprop0.5}.

\begin{proof}[Proof of Proposition \ref{xxprop0.5}]
Part (a) is Proposition \ref{xxprop4.6}(b). Parts (b) and (e)
are proved in Corollary \ref{xxcor4.7}. Part (c) follows  from
Lemmas \ref{xxlem4.9} and \ref{xxlem4.10}. Finally part (d) is 
Corollary \ref{xxcor4.8}.
\end{proof}

This is the end of Example \ref{xxex4.2}. 

There are many basic questions about double extensions
and here are some obvious ones. Some of 
which might be really easy, but haven't been solved at 
the time when this paper was written. 

\begin{question}
\label{xxque4.11} 
Let $R$ be a double extension of $A$.
\begin{enumerate}
\item
If $A$ is a domain, is then $R$ a domain?
\item
If $A$ is prime (or semiprime), is then $R$ 
prime (or semiprime)?
\end{enumerate}
\end{question}

\begin{question}
\label{xxque4.12} 
Let $R$ be a double extension of $A$. If $A$ is (right) 
noetherian, is then $R$ (right) noetherian?
\end{question}

If $R$ is an Artin-Schelter regular algebra of global dimension
dimension 4, then Question \ref{xxque4.12} has an affirmative
answer by the classification in \cite{ZZ}.

\begin{question}
\label{xxque4.13} 
Let $R$ be a double extension of $A$. 
Does the Auslander (respectively, Cohen-Macaulay) property 
pass from $A$ to $R$?
\end{question}

\begin{question}
\label{xxque4.14} 
Let $R$ be a double extension of $A$. 
If $A$ has finite (right) Krull dimension, does then
$R$ have finite (right) Krull dimension?
\end{question}

\section*{Acknowledgments}
James J. Zhang is supported by the US National Science Foundation
and the Royalty Research Fund of the University of Washington.

\providecommand{\bysame}{\leavevmode\hbox
to3em{\hrulefill}\thinspace}
\providecommand{\MR}{\relax\ifhmode\unskip\space\fi MR
}
\providecommand{\MRhref}[2]{%

\href{http://www.ams.org/mathscinet-getitem?mr=#1}{#2}
}
\providecommand{\href}[2]{#2}

\end{document}